\documentclass[hidelinks,onefignum,onetabnum]{siamart250211}

\usepackage{lipsum}
\usepackage{amsfonts}
\usepackage{graphicx}
\usepackage{epstopdf}
\usepackage{algorithmic}

\usepackage{amssymb}
\usepackage{mathrsfs}
\usepackage{amsmath, enumerate}
\usepackage{color}

\newtheorem{thm}{Theorem}[section]
\newtheorem{prop}[thm]{Proposition}
\newtheorem{cor}[thm]{Corollary}
\newtheorem{lem}[thm]{Lemma}
\newtheorem{defi}[thm]{Definition}

\newtheorem{cond}[thm]{Condition}

\numberwithin{equation}{section}
\numberwithin{thm}{section}

\newcommand{\qr}{{\mathbb{R}_\theta^d}}
\def\qd{\,{\mathchar'26\mkern-12mu d}}

\ifpdf
  \DeclareGraphicsExtensions{.eps,.pdf,.png,.jpg}
\else
  \DeclareGraphicsExtensions{.eps}
\fi

\newsiamremark{remark}{Remark}
\newsiamremark{hypothesis}{Hypothesis}
\crefname{hypothesis}{Hypothesis}{Hypotheses}
\newsiamthm{claim}{Claim}
\newsiamremark{fact}{Fact}
\crefname{fact}{Fact}{Facts}

\headers{Spectral Asymptotics for Quantized Derivatives}{Yongqiang Tian}

\title{Spectral Asymptotics for Quantized Derivatives on Quantum Euclidean Spaces\thanks{Submitted to the editors on May 16th.
\funding{This work was funded by the Fundamental Research Funds for the Central Universities of Central South University, China (no. 2022ZZTS0605). }}}

\author{Yongqiang Tian\thanks{School of Mathematics and Statistics, Central South University, Changsha, Hunan 
  (\email{tianyongqiang@csu.edu.cn}.)}
}

\usepackage{amsopn}

\ifpdf
\hypersetup{
pdftitle={Spectral Asymptotics for Quantized Derivatives on Quantum Euclidean Spaces},
pdfauthor={Yongqang Tian}
}
\fi

\begin{document}

\maketitle

\begin{abstract}
We obtain spectral asymptotics for the quantized derivatives of elements from the first-order homogeneous Sobolev space on the quantum Euclidean space.
extending an earlier result of McDonald, Sukochev and Xiong (Commun. Math. Phys. 2020). Our approach is based on a noncommutative Wiener-Ikehara Tauberian theorem and a recently developed $C^\ast$-algebraic version of pseudo-differential operator theory.
\end{abstract}

\begin{keywords}
Quantum Euclidean space, Quantized derivative, Spectral asymptotic
\end{keywords}

\begin{MSCcodes}
46G05, 47L10, 58B34
\end{MSCcodes}

\raggedbottom

\section{Introduction}\label{sec-introduction}
The investigation of quantized derivatives ties together various mathematical fields, such as noncommutative geometry, operator theory, real and harmonic analysis. Extensive studies of quantized derivatives in various settings can be found in the literature,  
for Riemannian spin manifolds we refer to \cite{Connes1988}, 
for quantum tori to \cite{MSX2019, SXZ2023}, for commutative and quantum Euclidean spaces to \cite{Peller2003, JW1982, RS1989, CST1994,LMSZ2017, FSZ2023, MSX2020, MSX2023},
see also \cite{CMSZ2017, CSZ2017} for interesting works on Julia sets and quasi-Fushian groups. In this paper, our main goal is to determine spectral asymptotics for quantized derivatives on the quantum Euclidean spaces. Of independent interest, we also establish spectral asymptotics for a certain class of abstract pseudo-differential operators.

Quantized derivatives can be considered in the context of a spectral triple. In Connes' noncommutative geometry \cite{Connes-ncdg-1985, Connes1994}, a spectral triple $(\mathcal{A},H,D)$ consists of a $*$-algebra $\mathcal{A}$ of bounded linear operators on a separable Hilbert space $H$ and an unbounded self-adjoint operator $D:{\rm dom}(D)\to H,$  such that for all $a\in\mathcal{A}:$
\begin{enumerate}[(i)]
\item $a:{\rm dom}(D)\to{\rm dom}(D),$ and$[D,a]$ has bounded extension on $H,$ 
\item $a(D+{\rm i})^{-1}$ is a compact operator.
\end{enumerate}
The \emph{quantized derivative} of an element $a \in \mathcal{A}$ is defined as the commutator $\qd a =[\mathrm{sgn}(D),a],$
where $\mathrm{sgn}(D)=D|D|^{-1}$ is the sign of $D.$ Central task in this topic is to study spectral properties of quantized derivatives. More precisely, to find conditions on $a$ such that $\qd a$ is contained in a certain ideal of compact operators. 
In Connes' quantized calculus \cite{Connes1995}, compact operators play the role of ``infinitesimals''. The order of an infinitesimal $T$ is described by the rate of the decay of its singular values:
\begin{equation}\label{def-of-sv}
\mu(n,T)=\inf\{\|T-F\|_{\infty}:\quad \mathrm{rank}(F)\leq n\},
\end{equation}
where $\|\cdot\|_\infty$ denotes the operator norm. In particular, an infinitesimal $T$ is said to be of order $\alpha$ if $\mu(n,T)=O((n+1)^{-\alpha}).$ If $\alpha=\frac1p>0,$ that is to say $T$ belongs to the Schatten-Lorentz ideal $\mathcal{L}_{p,\infty}.$  Recall that a spectral triple $(\mathcal{A}, H,D)$ is called $\mathcal{L}_{p,\infty}$-\emph{summable}, if $a(D+{\rm i})^{-1}\in\mathcal{L}_{p,\infty},$ for all $a\in\mathcal{A}.$ 
In this setting, spectral properties of quantized derivatives usually can be reinterpreted in terms of Sobolev spaces, Besov spaces, etc. as well as their noncommutative analogues. Of significant importance is endpoint weak Schatten property, which draws more attention originated \cite{Connes1988}. For several well-developed spectral triples mentioned above, those elements $a$ such that $\qd a\in\mathcal{L}_{p,\infty}$ are often captured by a first-order homogeneous Sobolev space. Moreover, through the Dixmier trace (or even, asymptotic of singular values) of $|\qd a|^p,$ one can approximately recover the corresponding Sobolev norm of $a,$ that is the quantity of first-order differentials.

When going back to the classical Euclidean spaces $\mathbb{R}^d$ ($d\geq2$), we can discuss it more precisely. Let us begin with some required notations.
For every $1\leq j\leq d,$ we denote by $\partial_j$ the $j$-th partial derivation on $\mathbb{R}^d.$
Denote $N_d=2^{\lfloor\frac d2\rfloor}.$ Let $\{\gamma\}_{j=1}^d\subset M_{N_d}(\mathbb{C})$ be Pauli matrices: 
$$\gamma_j^{\ast}=\gamma_j,\quad \gamma_j^2=1,\quad \gamma_j\gamma_k=-\gamma_k\gamma_j,\quad 1\leq j\neq k\leq d.
$$
The particular choice of such Pauli matrices is really unimportant, so we always fix one choice.
The Dirac operator associated to $\mathbb{R}^d$ is then defined by
$D=-{\rm i}\sum_{j=1}^{d}\gamma_j\otimes\partial_j,$
acting on the Hilbert space $\mathbb{C}^{N_d}\otimes L_2(\mathbb{R}^d)$ as an unbounded self-adjoint operator. In the setting, the quantized derivative of a suitable function $f$ on $\mathbb{R}^d$ is formally defined as $\qd f=[\mathrm{sgn}(D),1\otimes M_f],$ where $M_f$ is the (possibly unbounded) operator on $L_2(\mathbb{R}^d)$ of pointwise multiplication by $f.$

It is a classical result \cite{CRW1976} that $\qd f$ extends to a bounded operator on $\mathbb{C}^{N_d}\otimes L_2(\mathbb{R}^d)$ if and only if $f$ belongs to $BMO(\mathbb{R}^d),$ the space of functions with bounded mean oscillation. Recall that $\dot{W}^1_d(\mathbb{R}^d)$ is defined as the set of locally
integrable function $f$ with its
distributional gradient $\nabla f=(\partial_1 f,\cdots,\partial_d f)$ belonging to $L_d(\mathbb{R}^d, \mathbb{C}^d).$ It is also important to note that by the Poincar\'e
inequality, $\dot{W}^1_d(\mathbb{R}^d)\subset BMO(\mathbb{R}^d).$

In 1994, Connes-Sullivan-Teleman \cite{CST1994} (see the appendix there) stated that for any locally integrable function $f$ on $\mathbb{R}^d,$ one has \begin{equation}\label{connes' statement}
\qd f\in\mathcal{L}_{d,\infty} \Longleftrightarrow f\in \dot{W}^1_d(\mathbb{R}^d).
\end{equation} 
In 2017, Lord et al. \cite{LMSZ2017} gave an alternative and complete proof to \eqref{connes' statement} under the additional restriction that $f\in L_\infty(\mathbb{R}^d).$ Meanwhile, they also obtained a trace formula for quantized derivatives $\qd f$ of real-valued functions $f\in L_\infty(\mathbb
R^d)\cap \dot{W}^1_d(\mathbb{R}^d):$
\begin{equation}\label{trace formula0}
\varphi(|\qd f|^d)=c_d\|\nabla f\|_{L_d(\mathbb{R}^d,\mathbb{C}^d)},
\end{equation}
where $c_d>0$ is a constant and $\varphi$ can be any continuous normalized trace on $\mathcal{L}_{1,\infty}.$ For (singular) trace theory, we refer to \cite{LSZ2012} for more details.
In 2023, Frank-Sukochev-Zanin \cite{FSZ2023} removed the boundedness restriction but instead assumed that $f\in BMO(\mathbb{R}^d),$ and then presented a rigorous proof to the implications \eqref{connes' statement}, Furthermore, they derived the following spectral asymptotics for real-valued functions $f\in\dot{W}^1_d(\mathbb{R}^d),$
\begin{equation}\label{asymptotic on Euclidean space}
\lim_{n\to\infty}n^{\frac1d}\mu(n,\qd f)=c_d\|\nabla f\|_{L_d(\mathbb{R}^d,\mathbb{C}^d)}.
\end{equation}
We remark here that the trace formula result means that the sequence
$$\Big\{\sum_{k=2^n}^{2^{n+1}}\mu(k,|\qd f|^d)\Big\}_{n\geq0}$$ is almost convergent
in the sense of G.G. Lorentz \cite{L1948}, see \cite{SSUZ2015} or \cite[Section 5]{P2023}
for discussions. Since spectral asymptotic \eqref{asymptotic on Euclidean space} implies the precise convergence, it is certainly a significant refinement of the trace formula \eqref{trace formula0}. 

Inspired by spectral asymptotic result on Euclidean spaces (similar result also holds on quantum tori \cite{SXZ2023}), it is natural to consider its analogues on quantum Euclidean spaces. 
Fix a real and anti-symemtric $d\times d$ matrix $\theta.$ The quantum Euclidean space $L_\infty(\qr)$ is the von Neumann algebra on $L_2(\mathbb{R}^d)$ generated by a family of strongly continuous unitary operators $\{U(t)\}_{t\in\mathbb{R}^d}$ on $L_2(\mathbb{R}^d)$ satisfying the Weyl relation
\begin{equation}\label{weyl relation}
U(t)U(s)=e^{\frac{\rm i}{2}\langle t,\theta s\rangle}U(t+s),\quad s,t\in\mathbb{R}^d.
\end{equation}
All unexplained notations which we use below will be discussed
in Section \ref{nc plane}.

The algebra $L_\infty(\qr)$ admits a normal semifinite faithful trace $\tau_\theta,$ and the $L_p$-spaces associated to $(L_\infty(\qr),\tau_\theta)$ are simply denoted by $L_p(\qr).$ In the literature, differential operators on $L_\infty(\qr)$ are defined in a parametric-independent way. Let $\{D_j\}_{j=1}^d$ be the pointwise multiplication operators on $L_2(\mathbb{R}^d)$ defined by
$(D_j\xi)(t):=t_j\xi(t),$ for $\xi\in L_2(\mathbb{R}^d).$
With a slight abuse of notation, we also denote by $\{\partial_j\}_{j=1}^d$ the partial derivations on $L_\infty(\qr),$ which are spatially generated: $$\partial_jx:=[D_j,x],\quad x\in L_{\infty}(\qr), \quad j=1,\cdots,d.$$
Let $\dot{W}_d^1(\qr)$ denote the noncommutative homogeneous Sobolev space
$$\dot{W}_d^1(\qr):=\big\{x\in \mathcal{S}'(\qr):\|x\|_{\dot{W}_d^1(\qr)}:=\sum_{j=1}^{d}\|\partial_jx\|_{L_d(\qr)}<\infty\big\}.$$
The Dirac operator $\mathcal{D}$ associated to $L_{\infty}(\qr)$ is defined by setting
$\mathcal{D}:=\sum_{j=1}^{d}\gamma_j\otimes D_j.$
In particular, the datum $(1\otimes W_1^{\infty}(\qr),\mathbb{C}^{N_d}\otimes L_2(\mathbb{R}^d), \mathcal{D})$ forms a $\mathcal{L}_{d,\infty}$-summable spectral triple, see for instance \cite{SZ2023}.
For $x\in L_\infty(\qr),$ its quantized derivative is defined by setting 
$\qd x=[\mathrm{sgn}(\mathcal{D}),1\otimes x],$
which is a bounded operator on the Hilbert space $\mathbb{C}^{N_d}\otimes L_2(\mathbb{R}^d).$

Key a \emph{priori} result for quantized derivatives on quantum Euclidean spaces is established by Mcdonald-Sukochev-Xiong \cite{MSX2020, MSX2023}.
In the earlier paper \cite{MSX2020}, they provided a sufficient condition for $\qd x\in \mathcal{L}_{d,\infty},$ which is
\begin{equation}\label{sufficient condition}
x\in L_p(\qr)\cap\dot{W}_d^1(\qr),\quad \mbox{for~some}~d\leq p<\infty.
\end{equation} While the necessity of the condition $x\in \dot{W}_d^1(\qr)$ for $\qd x\in\mathcal{L}_{d,\infty}$ is discussed in Theorem 1.5 there.
Under the assumption \eqref{sufficient condition}, they obtained the following trace formula
\begin{equation}\label{trace formula}
\varphi(|\qd x|^d)=c_d\int_{\mathbb{S}^{d-1}}\tau_\theta\Big(\sum_{j=1}^{d}\big|\partial_jx-s_j\sum_{k=1}^ds_k\partial_kx\big|^2\Big)^{\frac d2}ds.
\end{equation}
Here, $\varphi$ is any continuous normalized trace on $\mathcal{L}_{1,\infty},$ the integral over unit sphere $\mathbb{S}^{d-1}$ is taken with respect to the surface measure $ds,$ and $s=(s_1,\cdots,s_d).$    
Recently, they found in \cite{MSX2023} that the restriction for $x\in \dot{W}_d^1(\qr)$ belonging to some $L_p(\qr)$ space is redundant. In fact, a combination of Theorems 4.8 and 5.3 in \cite{MSX2023} yields the following implications for $x\in L_\infty(\qr):$
\begin{equation}\label{MSX's statement}
\qd x\in \mathcal{L}_{d,\infty}\Longleftrightarrow x\in \dot{W}^1_d(\qr).
\end{equation}

\subsection{Main results} Our first result provides spectral asymptotic for quantized derivatives. We need a new semi-norm on  $\dot{W}_d^1(\qr)$:
\begin{equation}\label{equivalet seminorm} 
		|||x|||_{\dot{W}_d^1\left(\qr\right)}:=\Big\|\sum_{j=1}^d \gamma_j \otimes\big(\partial_j x\otimes 1-\sum_{k=1}^d \partial_k x \otimes \textbf{s}_k \textbf{s}_j\big)\Big\|_{L_d(M_{N_d}(\mathbb{C}) \bar{\otimes} L_{\infty}\left(\qr\right) \bar{\otimes} L_{\infty}\left(\mathbb{S}^{d-1}\right))}
\end{equation}
where $\{\mathbf{s}_j\}_{j=1}^d$ are coordinate functions on $\mathbb{S}^{d-1}.$
We will show in Section \ref{equivalence} that semi-norms $|||\cdot|||_{\dot{W}_d^1\left(\qr\right)}$ and $\|\cdot\|_{\dot{W}_d^1\left(\qr\right)}$ are in fact equivalent.

\begin{thm}\label{main1}
Assume $\det(\theta)\neq0.$ For any $x\in \dot{W}_d^1(\qr),$ we have 
\begin{equation}\label{asymptotic limit for qdx}
\lim_{n\rightarrow\infty} n^{\frac{1}{d}}\mu(n,\qd x)=\kappa_d|||x| ||_{\dot{W}_d^1\left(\qr\right)},\quad \kappa_d=\left(\frac{\Gamma(\frac{d}{2})}{2^{\frac{d}{2}+1}d\pi^d}\right)^{\frac1d}.
\end{equation}
\end{thm}
\begin{remark}
By a structure theorem for $L_{\infty}(\qr)$ (see \cite[Section 6]{LSZ2020} or \cite{GV1988} for detailed expositions), we have the following $\ast$-isomorphism:
$$L_{\infty}(\qr)\simeq \mathcal{L}_{\infty}(L_2({\mathbb{R}^{{\rm rank}(\theta)/2}}))\bar{\otimes}L_{\infty}(\mathbb{R}^{{\rm dim}({\rm ker}(\theta))}).$$
Then the non-degeneracy assumption $\det(\theta)\neq0$ leads to $L_{\infty}(\qr)\simeq \mathcal{L}_{\infty}(L_2(\mathbb{R}^{d/2})),$
and this restricts ourselves to Moyal planes rather than the full range of quantum Euclidean spaces. This is due to many technical reasons which will be explained later.
\end{remark}

Furthermore, we shall establish spectral asymptotics for a certain class of pseudo-differential operators.  The motivation of this generalization arises from a heuristic observation (Section \ref{sec-proof of main results}, Lemma \ref{msx 63 lemma}): every quantized derivative $\qd x$ (say, $x\in\mathcal{S}(\qr)$) can be approximately written as a linear combination of matrix-valued "pseudo-differential operators". 

The notion of a pseudo-differential operator on quantum Euclidean spaces defined in the traditional way appears in \cite{GJP2021}. However, for the particular purpose of this paper, we will employ instead the $C^*$-algebraic approach developed in \cite{SZ2018, MSZ2019}. Spectral asymptotics for classical pseudo-differential operators may date back to the work of Birman and Solomyak \cite{BS1970, BS1972}, see also \cite{Shub2001} and references therein. Recently, in the $C^\ast$-algebraic framework of pseudo-differential operators, similar results have been established for Euclidean spaces \cite{FSZ2023} and even for quantum tori \cite{MSX2023}.

Following \cite{MSZ2019}, the analogue of the class of classical pseudo-differential operators (of order $0$) is the $C^*$-algebra $\Pi,$ given as follows. 
\begin{defi}\label{def of Pi}\rm{}
Let $\pi_1:L_\infty(\qr)\to\mathcal{L}_{\infty}(L_2(\mathbb{R}^d)),$ $\pi_2:C(\mathbb{S}^{d-1})\to\mathcal{L}_{\infty}(L_2(\mathbb{R}^d))$ be two representations defined by setting
$\pi_1(x):=x$ and $(\pi_2(g)\xi)(t):=g(\frac{t}{|t|})\xi(t),$ $\xi\in L_2(\mathbb{R}^d).$
Let
$\mathcal{A}_1=C_0(\qr)+\mathbb{C},$ $ \mathcal{A}_2=C(\mathbb{S}^{d-1}),$
and let $\Pi=\Pi(\mathcal{A}_1,\mathcal{A}_2)$ be the $C^\ast$-subalgebra of $\mathcal{L}_{\infty}(L_2(\mathbb{R}^d))$ generated by the algebras $\pi_1(\mathcal{A}_1)$ and $\pi_2(\mathcal{A}_2).$ 
\end{defi}
Recalling \cite[Theorem 3.3]{MSZ2019}, there exists a unique continuous $\ast$-homomorphism $\mathrm{sym}:\Pi\to\mathcal{A}_1\otimes_{\min}\mathcal{A}_2=C(\mathbb{S}^{d-1},C_0(\qr)+\mathbb{C})$ such that for all $x\in\mathcal{A}_1$ and $g\in\mathcal{A}_2:$
\begin{equation}\label{tensor seperating}
\mathrm{sym}(\pi_1(x))=x\otimes1,\quad \mathrm{sym}(\pi_2(g))=1\otimes g.
\end{equation}
Let $\mathcal{K}(L_2(\mathbb{R}^d))$ denote the algebra of compact operators on $L_2(\mathbb{R}^d).$ By its original construction, $\mathrm{sym},$ is a composition of the Calkin quotient map $q: \mathcal{L}_{\infty}(L_2(\mathbb{R}^d))\to\mathcal{L}_{\infty}(L_2(\mathbb{R}^d))/\mathcal{K}(L_2(\mathbb{R}^d))$ with a $C^\ast$-homomorphism, therefore, vanishes on compact operators. Namely,
\begin{equation}\label{compact vanishing}
\mathrm{sym}(T)=0,\quad \mbox{for~all~} T\in\Pi\cap\mathcal{K}(L_2(\mathbb{R}^d)).
\end{equation}

Our second result is Theorem \ref{main2} below, it can be regarded as a noncommutative version of \cite[Theorem 1.5]{FSZ2023}.
Recall that an operator $T\in \Pi$ is said to be \emph{compactly supported on the right}, if $T=T\pi_1(x)$ for some $x$ in the Schwartz space $\mathcal{S}(\qr).$ A fundamental example for such $T$ can be $T=\pi_2(g)\pi_1(x),$ where $g\in C(\mathbb{S}^{d-1})$ and $x\in\mathcal{S}(\qr)$ is an idempotent element. Quite unlike in classical Schwartz spaces, for $\det(\theta)\neq0$ there are plenty of non-trivial idempotent elements in $\mathcal{S}(\qr),$ see e.g. \cite{Gay2004}. 
In the following, $\Delta_{\theta}$ stands for the Laplacian associated to $L_{\infty}(\qr),$
$(\Delta_{\theta}\xi)(t):=-|t|^2\xi(t),$ $ \xi\in L_2(\mathbb{R}^d).$
Here, we use the subscript $\theta$ to distinguish it from the classical Laplacian $\Delta.$
\begin{thm}\label{main2}Assume $\det(\theta)\neq0.$
If $T\in\Pi$ is compactly supported on the right, then
$$\lim _{n \rightarrow \infty} n^{\frac{1}{d}} \mu(n, T(1-\Delta_{\theta})^{-\frac12})=\kappa_d\| \operatorname{sym}(T)\|_{L_d( L_{\infty}\left(\qr\right) \bar{\otimes} L_{\infty}\left(\mathbb{S}^{d-1}\right))}.$$
Here, the constant $\kappa_d$ is given as in Theorem \ref{main1}.
\end{thm}

Essential difficulty in studying spectral asymptotics on $L_{\infty}(\qr)$ comes mainly from the non-compactness of the (resolvent of the) Dirac operator $\mathcal{D}.$ Note that $(1+\mathcal{D}^2)^{-\frac12}=1\otimes(1-\Delta_{\theta})^{-\frac12}.$ The multiplication operator $(1-\Delta_{\theta})^{-\frac12}$ on $L_2(\mathbb{R}^d)$ with symbol $(1+|t|^2)^{-\frac12}$ is clearly non-compact. This leads to an unsatisfactory fact that one generally cannot determine spectral asymptotic of an operator $T(1-\Delta_{\theta})^{-\frac12}$ with $T\in\Pi$ through a direct computation of its singular values. 
The method in \cite{FSZ2023} deals with intervals in $\mathbb{R}^d,$ there seems no possibility in adjusting it to our present case. An available method (might be the only available one) in a noncommutative setting was given by Sukochev-Xiong-Zanin \cite{SXZ2023}. Where they computed the singularities of operator zeta functions of the form ${\rm Tr}(A^zB^z)$ and then applied a noncommutative Wiener-Ikehara Tauberian theorem \cite[Section 5.4]{SZ2023} to establish the fundamental spectral asymptotic. Our method for quantum Euclidean spaces resembles that of \cite{SXZ2023} for quantum tori, however, due to non-compactness of Dirac operators, this turns out to be a much involved task. Many nice algebraic and analysis properties for operators are required to overcome difficulties in dealing with the operator zeta function ${\rm Tr}(A^zB^z),$ this partially explains why we limit our consideration on Moyal planes. 

The rest of this paper is organized as follows. In Section \ref{sec-preliminary}, we review background material concerning operator ideals, asymptotics for singular values and
analysis on quantum Euclidean spaces. Section \ref{sec-commutator} is devoted to prove useful
commutator estimates. In Section \ref{sec-mixed trace formula}, we establish a trace formula for convolution-product nuclear operators. In Section \ref{sec-residue}, we use this trace formula to compute singularities of operator zeta functions. In Section \ref{sec-asymptotic special case}, we establish an fundamental spectral asymptotics, which contains the main difficulty. In the final section, we complete the proofs of our main results\textemdash Theorems \ref{main1} and \ref{main2}, additionally, we prove the equivalence of semi-norms $|||\cdot|||_{\dot{W}_d^1(\qr)}$ and $\|\cdot\|_{\dot{W}_d^1(\qr)}.$

\textbf{Notation:} In what follows, set $\mathbb{N}=\{1,2,3,\cdots\}$ and  $\mathbb{Z}_+=\{0,1,2,\cdots\}.$ For a multi-index $\alpha\in\mathbb{Z}_+^d,$ define its length as $|\alpha|_1:=\alpha_1+\alpha_2+\cdots+\alpha_d.$ By $\chi_{E}$ we denote the indicator function of a set $E.$ By $\mathrm{tr}$ we denote the standard matrix trace.
We write $A \lesssim_{\xi} B$ if $A \leq  c_{\xi} B$ for some constant $c_{\xi}$
only depending on $\xi.$

\section{Preliminaries}\label{sec-preliminary}
\subsection{Singular values and Schatten ideals}\label{operator notation subsection}
In this subsection, we overview some basic facts about singular values and operator ideals, see \cite{BS1987, LSZ2012} for more details.

Fix a complex separable Hilbert space $H.$
Let $\mathcal{L}_{\infty}(H)$ denote the algebra of bounded linear operators on $H,$ equipped with the operator norm $\|\cdot\|_\infty.$  The canonical trace on $\mathcal{L}_{\infty}(H)$ is denoted by $\mathrm{Tr}.$ Let $\mathcal{K}(H)$ denote the set of compact operators on $H.$ Given $T\in \mathcal{K}(H),$ recall that  by $\mu(T)=\{\mu(n,T)\}_{n=0}^\infty$ we denote the sequence of singular values determined by \eqref{def-of-sv}. Due to the min-max principle, $\mu(n,T)$ coincides with the $(n+1)$-th largest eigenvalues of $|T|=\sqrt{ T^\ast T}.$ 

For $T\in\mathcal{L}_{\infty}(H),$ it is very convenient to define the singular value function
\begin{equation*}
\mu(t,T)=\inf\{\|T-R\|_{\infty}:\quad \mathrm{rank}(R)\leq t\},\quad t\geq0.
\end{equation*}
Below, we record some key properties of singular values (see e.g. \cite[Section 2.3]{LSZ2012}):
$$\mu(t,UTV)\leq\|U\|_{\infty}\mu(t,T)\|V\|_{\infty},$$
$$
\mu(s+t,S+T)\leq\mu(s,S)+\mu(t,T),
\quad
\mu(s+t,ST)\leq \mu(s,S)\mu(t,T)
$$
and
$\mu(t,T)=\mu(t,T^*)=\mu(t,|T|^{\alpha})^{\frac1{\alpha}},$ for all $\alpha>0.$
For $0<p<\infty,$ we denote by $\mathcal{L}_p(H)$ the Schatten-von Neumann ideal  
$\mathcal{L}_p(H)=\{T\in\mathcal{K}(H):\mathrm{Tr}(|T|^p)<\infty\},$
equipped with the $p$-norm
$\|T\|_p= \big(\sum_{k=0}^\infty\mu(k,T)^p\big)^{1/p}.$
If $p\geq1,$ then $\mathcal{L}_p(H)$ is a Banach ideal. Similarly, for $0<p<\infty$ we denote by $\mathcal{L}_{p,\infty}(H)$ the Schatten-Lorentz ideal
$\mathcal{L}_{p,\infty}(H)=\{T\in \mathcal{K}(H):\sup_{t>0}t^{\frac1p}\mu(t,T)<\infty\},$
equipped with the quasi-norm
$\|T\|_{p,\infty}:=\sup_{t>0}t^{\frac1p}\mu(t,T).$
If $p>1,$ then $\mathcal{L}_{p,\infty}(H)$ is a quasi-Banach ideal.
We will abbreviate $\mathcal{L}_{p}(H)$ as $\mathcal{L}_p,$ $\mathcal{L}_{p,\infty}(H)$ as $\mathcal{L}_{p,\infty}$ and so on, when the Hilbert space is
clear from context.

We make frequent use of the following quasi-triangle inequality:
$
\|S+T\|_{p,\infty}\leq 2^{\frac1p}\big(\|S\|_{p,\infty}+\|T\|_{p,\infty}\big).
$
Similarly, for $\frac{1}{r}=\frac{1}{p}+\frac{1}{q}$ we have the H\"older inequality
$
\|ST\|_{r,\infty}\leq 2^{\frac1r}\|S\|_{p,\infty}\|T\|_{q,\infty}.
$
Recently, the best constants for H\"older inequalities were found by Sukochev-Zanin \cite{SZ2021}.

Let $(\mathcal{L}_{p,\infty})_0$ be the closure of all finite rank operators in $\mathcal{L}_{p,\infty}.$ Equivalently,
$(\mathcal{L}_{p,\infty})_0=\{V\in\mathcal{L}_{p,\infty}:\ \lim_{t\to\infty}t^{\frac1p}\mu(t,V)=0 \}.$
The ideal $(\mathcal{L}_{p,\infty})_0$ is also termed as the separable part of $\mathcal{L}_{p,\infty}.$
It is important to note that $\mathcal{L}_{p,\infty}\subset(\mathcal{L}_{q,\infty})_0$ whenever $p<q.$ Note also that,
by the H\"older inequality we have
$\mathcal{K}\cdot\mathcal{L}_{p,\infty}\subseteq(\mathcal{L}_{p,\infty})_0$
and
$ (\mathcal{L}_{p,\infty})_0\cdot\mathcal{L}_{q,\infty}\subseteq(\mathcal{L}_{r,\infty})_0,$ $ \frac1r=\frac1p+\frac1q.$ 
\subsection{Abstract limit lemmas}
Fix a parameter $p\in(0,\infty)$ in this subsection.  The following well-known result is due to Ky Fan, see e.g. \cite[Section 11.6]{BS1987}.

\begin{lem}\label{bs sep lemma}
Let $T\in\mathcal{L}_{p,\infty}$ and $V\in(\mathcal{L}_{p,\infty})_0.$ If $\lim_{t\to\infty}t^{\frac1p}\mu(t,T)$ exists, then
$$\lim_{t\to\infty}t^{\frac1p}\mu(t,T+V)=\lim_{t\to\infty}t^{\frac1p}\mu(t,T).$$
\end{lem}

The next lemma quoted from \cite{SXZ2023} will be repeatedly used.  Recall that the operators $\{T_k\}_{k=1}^n\subset \mathcal{L}_{\infty}$ are called \emph{pairwise orthogonal} if 
$T_kT_l=T_k^\ast T_l=0,$ $k\neq l.$

\begin{lem}\label{bs dirsum lemma}
Let $\{T_k\}_{k=1}^n\subset\mathcal{L}_{p,\infty}$ be a sequence of pairwise orthogonal operators. Suppose that for all $1\leq k\leq n,$ there exist the limits
$\lim_{t\to\infty}t^{\frac1p}\mu(t,T_k)=a_k,$ then
$$\lim_{t\to\infty}t^{\frac1p}\mu(t,\sum_{k=1}^nT_k)=\big(\sum_{k=1}^na_k^p \big)^{\frac1p}.$$
\end{lem}

The following result is due to Birman-Solomyak \cite[Theorem 4.1, Remark 4.2]{BS1970}.
\begin{lem}\label{another limit lemma}
Let $T\in\mathcal{L}_{p,\infty}$ and let $\{T_n\}_{n=1}^{\infty}\subset\mathcal{L}_{p,\infty}.$ Suppose that
\begin{enumerate}[{\rm (i)}]
\item\label{aala} for every $n\geq 1$ we have
$\lim _{t\to\infty}t^{\frac1p}\mu(t,T_n)=a_n;$
\item\label{aalb} ${\rm dist}_{\mathcal{L}_{p,\infty}}\big(T_n-T,(\mathcal{L}_{p,\infty})_0\big)\to0$ as $n\to\infty.$
\end{enumerate}	
Under those assumptions, we have
$$\lim_{t\to\infty}t^{\frac1p}\mu(t,T)=\lim_{n\to\infty}c_n.$$
\end{lem}
We end this subsection by a useful distance formula in $\mathcal{L}_{p,\infty},$ see \cite[p. 267]{BS1987}.
\begin{equation}\label{distance formula}
\mathrm{dist}_{\mathcal{L}_{p,\infty}}(T,(\mathcal{L}_{p,\infty})_0)=\limsup_{t\to\infty}t^{\frac1p}\mu(t,T).
\end{equation}

\subsection{Quantum Euclidean space:  algebra}\label{nc plane}

There are many approaches in the literature to define the quantum Euclidean spaces, see  \cite{Gay2004, R1993} and \cite[Section 6]{LSZ2020}. 
For simplicity, we adopt the one suggested in \cite{LSZ2020}. 

Let us consider a concrete realization for the Weyl relation \eqref{weyl relation}, namely, we set
\begin{equation}\label{definition of unitary family}
(U_\theta(t)\xi)(u):=e^{-\frac{\rm i}{2}\langle t,\theta u\rangle}\xi(u-t),\quad\xi\in L_2(\mathbb{R}^d),\quad t,u\in\mathbb{R}^d.
\end{equation}
Since $\theta$ is anti-symmetric, we have
$ U_\theta(t)^\ast=U_\theta(t)^{-1}=U_\theta(-t),\quad t\in\mathbb{R}^d.$
\begin{defi}The von Neumann algebra on $L_2(\mathbb{R}^d)$ generated by $\{U_\theta(t)\}_{t\in\mathbb{R}^d}$ introduced in \eqref{definition of unitary family}, is called the \emph{quantum Euclidean space} and denoted by $L_\infty(\mathbb{R}^d_\theta).$
\end{defi}

Everywhere below, we focus exclusively on the non-degenerate case $\det(\theta)\neq0.$ 
Due to anti-symmetry and non-degeneracy of the matrix $\theta,$ we have
$\det(\theta)=\det(\theta^{\mathrm{\ast}})=\det(-\theta)=(-1)^d\det(\theta)\neq0.$
So, $d$ is automatically an even integer.
For every such $\theta,$ the algebra $L_\infty(\qr)$ is $\ast$-isomorphic to $\mathcal{L}_{\infty}(L_2(\mathbb{R}^{\frac d2})),$ we refer to \cite[Chapter 4]{LMSZ2023} for further details. 
Denote $r_\theta: L_\infty(\qr)\to\mathcal{L}_{\infty}(L_2(\mathbb{R}^{\frac d2}))$ for the corresponding  $\ast$-isomorphism, and let $\mathrm{Tr}$ be the standard trace on $\mathcal{L}_{\infty}(L_2(\mathbb{R}^{\frac{d}{2}})).$
Following \cite[Section 4.2]{LMSZ2023}, we can equip $L_\infty(\qr)$ with a n.s.f. trace
$\tau_\theta:=\big((2\pi)^{d}\det(\theta)\big)^{\frac12}\cdot\mathrm{Tr}\circ r_\theta.
$
\begin{defi}
The $L_p$-spaces on $\qr$ are denoted by $L_p(\qr),$ $0<p<\infty.$ Namely,
$L_p(\qr)=\{x\in L_\infty(\qr):\|x\|_{L_p(\qr)}:=\tau_\theta(|x|^p)^{\frac{1}{p}}<\infty\}.$
\end{defi}
Since $r_\theta$ is an $\ast$-isomorphism, it is immediate that $r_\theta(L_p(\qr))=\mathcal{L}_p(L_2(\mathbb{R}^{\frac d2})).$ This implies that the $L_p$-spaces on $\qr$ are nested, i.e., $L_p(\qr)\subset L_q(\qr)$ whenever $p<q.$ From the definition of $\tau_\theta,$ we indeed have
$\|x\|_q\leq\big((2\pi)^{d}\det(\theta)\big)^{\frac{1}{2q}-\frac{1}{2p}}\|x\|_p,$ for all $ x\in L_p(\qr).$
\begin{defi}
Let $C_0(\qr)$ be the $C^\ast$-algebra of all 
$\tau_\theta$-compact operators in $L_\infty(\qr).$ Equivalently, $C_0(\qr)=r^{-1}_\theta(\mathcal{K}(L_2(\mathbb{R}^{\frac d2})).$
\end{defi}

For $f\in\mathcal{S}(\mathbb{R}^d),$ denote by $U_\theta(f)\in L_\infty(\qr)$ the operator
\begin{equation*}
U_\theta(f)\xi:=\int_{\mathbb{R}^d}f(t) U_\theta(t)\xi dt,\quad \xi\in L_2(\mathbb{R}^d).
\end{equation*}
Here, the integral is the $L_2(\mathbb{R}^d)$-valued Bochner integral.
By \cite[Lemma 4.2.14]{LMSZ2023} the operator $U_\theta(f)$ is in $L_1(\qr)$ and its trace can be explicitly computed:
\begin{equation}\label{trace of schwartz function}
\tau_\theta(U_\theta(f))=(2\pi)^d f(0).
\end{equation}
\begin{defi}
The Schwartz space $\mathcal{S}(\qr)$ on the quantum Euclidean space is defined as the image of the classical Schwartz space under $U_\theta,$ that is,
$\mathcal{S}(\qr):=U_\theta(\mathcal{S}(\mathbb{R}^d)).$
\end{defi}

The Schwartz space $\mathcal{S}(\mathbb{R}^d
)$ can be reinterpreted as a twisted Fr\'echet algebra, which we now explain.
For $f,g\in\mathcal{S}(\mathbb{R}^d),$ define
$f^\ast:=\bar{f}(-t)$ as the \emph{involution} of $f,$ and set $$\quad (f\star_\theta g)(t):=\int_{\mathbb{R}^d}e^{\frac{\rm i}{2}\langle t,\theta s\rangle}f(t-s)g(s)ds=(U_\theta(f)g)(t).$$ 
The bilinear map $(f,g)\mapsto f\star_\theta g$ is called the \emph{twisted convolution} of the Schwartz functions $f$ and $g.$ It is obvious that $f^\ast\in\mathcal{S}(\mathbb{R}^d)$ and  $f\star_\theta g\in \mathcal{S}(\mathbb{R}^d).$
Consequently, the triple $(\mathcal{S}(\mathbb{R}^d),\ast, \star_\theta)$ under its original Fr\'echet topology becomes a Fr\'echet $\ast$-algebra.

Further, a direct computation shows
$$U_\theta(f+g)=U_\theta(f)+U_\theta(g),\quad U_\theta(f)^\ast=U_\theta(f^\ast) ,\quad U_\theta(f)U_\theta(g)=U_\theta(f\star_\theta g).$$
Having these properties at hand, we now infer that the surjection $U_\theta:\mathcal{S}(\mathbb{R}^d)\to\mathcal{S}(\qr)$ is in fact an algebraic $\ast$-isomorphism. This is done by knowing that $U_\theta$ is injective, see e.g. \cite[p 500]{MSX2020}.  

For any $f\in\mathcal{S}(\mathbb{R}^d),$ using \eqref{trace of schwartz function} we compute
$$\tau_\theta(U_\theta(f)^\ast U_\theta(f))=\tau_\theta(U_\theta(f^{\ast}\star_\theta f))=(2\pi)^d\cdot(f^{\ast}\star_\theta f)(0)=(2\pi)^d\|f\|_2^2.$$
Thus, $\|U_\theta(f)\|_2=(2\pi)^{\frac{d}{2}}\|f\|_2.$ Actually, the map $(2\pi)^{-\frac{d}{2}}U_\theta$ extends to an isometry from $L_2(\mathbb{R}^d)$ onto $L_2(\qr),$ see e.g. \cite[Section 4.2]{LMSZ2023}.
\begin{prop}\label{density}
$\mathcal{S}(\qr)$ is norm-dense in $C_0(\qr).$
\end{prop}
\begin{proof}
Since $\mathcal{S}(\mathbb{R}^d)$ is dense in $L_2(\mathbb{R}^d),$ and $(2\pi)^{-\frac{d}{2}}U_\theta:L_2(\mathbb{R}^d)\to L_2(\qr)$ is an isomorphism, it follows that $\mathcal{S}(\qr)=U_\theta(\mathcal{S}(\mathbb{R}^d))$ is dense in $L_2(\qr).$
Moreover, since the Hilbert-Schmidt class $\mathcal{L}_2(L_2(\mathbb{R}^{\frac{d}{2}}))$ is norm-dense in $\mathcal{K}(L_2(\mathbb{R}^{\frac{d}{2}}))$ and so, through the $\ast$-isomorphism $ r^{-1}_\theta:\mathcal{L}_{\infty}(L_2(\mathbb{R}^{\frac{d}{2}}))\to L_\infty(\qr),$ we conclude that $L_2(\qr)$ is norm-dense in $C_0(\qr).$
Altogether, $\mathcal{S}(\qr)$ is norm-dense in $C_0(\qr).$
\end{proof}
\subsection{Noncommutative Euclidean space: calculus}In this subsection, we review basic properties of several functions spaces on $\qr.$ 
Recall that by $D_j,$ $j=1,\cdots,d$ we denote the multiplication operators 
$(D_j \xi)(t)=t_j \xi(t),$ $ \xi\in L_2(\mathbb{R}^d)$
defined on the domain $\mathrm{dom}(D_j)=\{\xi\in L_2(\mathbb{R}^d):\xi\in L_2(\mathbb{R}^d,t^2_jdt)\}.$ 
Clearly, $\mathcal{S}(\mathbb
R^d)$ is a common core of $\{D_j\}_{j=1}^d.$
Fix $j$ and $t\in\mathbb{R}^d.$ Note that the unitary operator $U_\theta(t)$ preserves $\mathrm{dom}(D_j),$ thus for $\xi\in \mathcal{S}(\mathbb{R}^d),$ we may compute 
$
[D_j,U_\theta(t)]\xi=t_jU_\theta(t)\xi.
$ It is readily that the commutator $[D_j,U_\theta(t)]$ extends to a bounded operators on $L_2(\mathbb{R}^d).$ There is a crucial fact \cite[Proposition 6.12]{LSZ2020}, which asserts that: if $x\in L_{\infty}(\qr)$ and  if $[D_j,x]$ extends to a bounded operator on $L_2(\mathbb{R}^d),$ then $[D_j,x]\in L_{\infty}(\qr).$
In this case, we will write $\partial_jx:=[D_j,x],$ and then call $\partial_jx$ the \emph{$j$-th partial derivative} of $x.$ 
Generally, for $\alpha\in\mathbb{Z}_+^d$ we consider the mixed partial derivation
$
\partial^\alpha:=\partial_1^{\alpha_1}\cdots\partial_d^{\alpha_d}$
with the domain 
$\mathrm{dom}(\partial^{\alpha}):=\{x\in L_\infty(\qr): \partial^{\alpha}x\in L_\infty(\qr)\}.$
As usual, $\partial^0x:=x.$

Let $f\in\mathcal{S}(\mathbb{R}^d).$ It is straightforward that
$$\partial_j(U_\theta(f))=\int_{\mathbb{R}^d}t_jf(t)U_\theta(t)dt\in\mathcal{S}(\qr),\quad 1\leq j\leq d.$$
Inductively, for every $\alpha\in\mathbb{Z}_+^d$ we have $\partial^{\alpha}(U_\theta(f))=U_\theta(f_\alpha)\in \mathcal{S}(\qr),$ $ f_\alpha(t):=t^{\alpha_1}_1\cdots t^{\alpha_d}_d f(t).$
Since $\alpha$ can be arbitrary, we have showed that 
$\partial^{\alpha}(\mathcal{S}(\qr))\subset\mathcal{S}(\qr),$ for all $\alpha.$
Next, we introduce the noncommutative Sobolev spaces on $\qr.$
\begin{defi}
Let $k\in\mathbb{Z}_+$ and let $1\leq p\leq\infty.$ The Sobolev space $W_p^k(\qr)$ is defined by setting
$W_p^k(\qr):=\{x\in L_p(\qr): \partial^{\alpha}x\in L_p(\qr), ~\forall~|\alpha|_1\leq k \},$
equipped with the norm
$$
\|x\|_{W^k_p(\qr)}=\sum_{|\alpha|_1 \leq k}\left\|\partial^\alpha x\right\|_p.
$$
With this norm $W^k_p(\qr)$ is a Banach space \cite{MSX2020}.
Similarly, the homogeneous Sobolev space $\dot{W}_p^k(\qr)$ is defined by setting 
$\dot{W}_p^k(\qr):=\{x\in \mathcal{S}'(\qr): \partial^{\alpha}x\in L_p(\qr), ~\forall~|\alpha|_1=k\},$
equipped with the semi-norm
$$
\|x\|_{\dot{W}_p^k(\qr)}=\sum_{|\alpha|_1=k}\left\|\partial^\alpha x\right\|_p.
$$
\end{defi}
Recalling that $U_\theta(f)\in L_1(\qr)$ for all $f\in\mathcal{S}(\mathbb{R}^d),$ it follows that $\mathcal{S}(\qr)$ is a common subspace of $W^k_p(\qr)$ for all $1\leq p\leq\infty.$
\subsection{The Schwartz space $\mathcal{S}(\qr)$ revisited}
In this subsection, we shall recall a matrix description of $\mathcal{S}(\qr),$ it is due to Gracia-Bond\'ia-V\'arilly \cite{GV1988}. 
\begin{defi}
Let $d\in2\mathbb{N}.$ A square-summable double sequence $a\in\ell_2(\mathbb{Z}_+^{\frac d2}\times \mathbb{Z}_+^{\frac d2})$ is said to be \emph{rapidly decreasing} if, for every $m\in\mathbb{Z}_+,$  
$$r_m(a):=\Big(\sum_{k,l\in\mathbb{Z}_+^{d/2}}(|k|+1)^{2m}(|l|+1)^{2m}|a_{k,l}|^2\Big)^{\frac12}<\infty.$$
Let $\mathbf{S}_d\subset\ell_2(\mathbb{Z}_+^{\frac d2}\times \mathbb{Z}_+^{\frac d2})$ denote the space of rapidly decreasing square-summable double sequences.  In
addition, we equip $\mathbf{S}_d$ with the involution
$$a^{\ast}=\left(\overline{ a_{k,j}}\right)_{j,k\in \mathbb{Z}_+^{d/2}},\quad a=(a_{j,k})_{j,k\in \mathbb{Z}_+^{d/2}}\in\mathbf{S}_d$$ and matrix product
$$a\cdot b:=\big(\sum_{k\in\mathbb{Z}_+^{d/2}}a_{j,k}b_{k,l}\big)_{j,l\in \mathbb{Z}_+^{d/2}},\quad a,b\in\mathbf{S}_d.$$
\end{defi}
Note that for any $j,l\in\mathbb{Z}_+^{d/2},$
$$ \big|\sum_{k\in\mathbb{Z}_+^{d/2}}a_{j,k}b_{k,l}\big|^2\leq \sum_{k\in\mathbb{Z}_+^{d/2}}|a_{j,k}|^2\cdot\sum_{k\in\mathbb{Z}_+^{d/2}}|b_{k,l}|^2.$$
Thus,
$$r_m(a\cdot b)^2\leq \sum_{k\in\mathbb{Z}_+^{d/2}}(|j|+1)^{2m}|a_{j,k}|^2\cdot\sum_{k\in\mathbb{Z}_+^{d/2}}(|l|+1)^{2m}|b_{k,l}|^2\leq  r_m(a)^2r_m(b)^2.$$
Equipped with the family of semi-norms $\{r_m\}_{m\in\mathbb{Z}_+},$ ${\bf S}_d$ forms a Fr\'echet space. Moreover, with above involution and matrix product, ${\bf S}_d$ is in fact a Fr\'echet $\ast$-algebra.
The following is due to
\cite[Proposition 2.5]{Gay2004}, see also \cite[Theorem 6]{GV1988}.
\begin{prop}\label{structure of schwartz spaces}
There exists a sequence $\{f^{\theta}_{k,l}\}_{k,l\in\mathbb{Z}_+^{d/2}}\subset\mathcal{S}(\mathbb{R}^d)$ of Schwartz functions  such that
\begin{enumerate}[{\rm (i)}]
\item $f^{\theta}_{k_1,l_1}\star_\theta f^{\theta}_{k_2,l_2}=\delta_{l_1,k_2}f^{\theta}_{k_1,l_2}$ and $\overline{f^{\theta}_{k,l}}=f^{\theta}_{l,k},$
\item $\langle f^{\theta}_{k_1,l_1},f^{\theta}_{k_2,l_2} \rangle =\delta_{k_1,k_2}\delta_{l_1,l_2}.$
\item the mapping $\rho_{\theta}:{\bf S}_d\to \mathcal{S}(\mathbb{R}^d)$ given by
$$\rho_{\theta}(a)= \sum_{k,l\in\mathbb{Z}_+^{d/2}}a_{k,l}f^{\theta}_{k,l},\quad a\in\mathbf{S}_d$$
is a $\ast$-isomorphism of Fr\'echet algebras. 
\end{enumerate}
Moreover, the sequence $\{f^{\theta}_{k,l}\}_{k,l\in\mathbb{Z}_+^{d/2}}$ forms an orthonormal basis for $L_2(\mathbb{R}^d).$
\end{prop}
\begin{cor}\label{star algebra iso}
The mapping $U_{\theta}\circ \rho_{\theta}:{\bf S}_d\to\mathcal{S}(\qr)$ given by
$$
a\mapsto\sum_{k,l\in\mathbb{Z}_+^{d/2}}a_{k,l}U_\theta(f^{\theta}_{k,l})
$$
leads to an isomorphism between the $\ast$-algebras $\mathbf{S}_d$ and $\mathcal{S}(\qr).$ Here, the series converges in the sense of $L_1(\qr).$
\end{cor}

\begin{prop}\label{root lemma}
If $0\leq x\in\mathcal{S}(\qr),$ then $x^p\in\mathcal{S}(\qr)$ for all $p>0.$
\end{prop}
\begin{proof}
The result for $d=2$ is explicitly established in \cite{PSVZ2019} (see Lemma 6.9 and Corollary 6.13 there). The result for $d>2$ follows \emph{mutatis mutandis}. For the reader's convience, we sketch the proof for $p=\frac12$ along the lines of \cite[Lemma 6.9]{PSVZ2019}. By assumption, Proposition \ref{structure of schwartz spaces} and Corollary \ref{star algebra iso}, there exist a double sequence $0\leq a=(a_{k,l})\in{\bf S}_d$ such that $x=U_{\theta}(\rho_{\theta}(a)).$ So it suffices to show $b:=a^{\frac12}\in{\bf S}_d.$ Let $\{e_{k}\}_{k\in\mathbb{Z}_+^{d/2}}\subset\ell_2(\mathbb{Z}_+^{d/2})$ be standard basis vectors. We compute
$$\begin{aligned}
|b_{k,l}|^2&=|\langle b e_k,e_l\rangle|\cdot|\langle e_k,be_l\rangle|\\
&\leq\|be_k\|_{\ell_2}\|be_k\|_{\ell_2}\\
&=|\langle a e_k,e_k\rangle|^{\frac12}\cdot|\langle ae_l,e_l\rangle|^{\frac12}=|a_{k,k}|^{\frac12}\cdot |a_{l,l}|^{\frac12}.
\end{aligned} $$
Thus, for each $m\geq0,$ we have
$$\begin{aligned}
r_m(a^{\frac12})^2&=\sum_{k,l\in\mathbb{Z}_+^{d/2}}(|k|+1)^{2m}(|l|+1)^{2m}|b_{k,l}|^2\\
&\leq\sum_{k,l\in\mathbb{Z}_+^{d/2}}(|k|+1)^{2m}(|l|+1)^{2m}|a_{k,k}|^{\frac12} |a_{l,l}|^{\frac12}=\Big(\sum_{k\in\mathbb{Z}_+^{d/2}}(|k|+1)^{2m}|a_{k,k}|^{\frac12}\Big)^2.
\end{aligned}$$
By the H\"older inequality, we obtain
$$\begin{aligned}
\sum_{k\in\mathbb{Z}_+^{d/2}}(|k|+1)^{2m}|a_{k,k}|^{\frac12}&=\sum_{k\in\mathbb{Z}_+^{d/2}}(|k|+1)^{-\frac{d}{2}}\cdot(|k|+1)^{2m+\frac{d}{2}}|a_{k,k}|^{\frac12}\\&\leq \Big(\sum_{k\in\mathbb{Z}_+^{d/2}}(|k|+1)^{-\frac{2d}{3}}\Big)^{\frac34} \cdot\Big(\sum_{k\in\mathbb{Z}_+^{d/2}}(|k|+1)^{8m+2d}|a_{k,k}|^2\Big)^{\frac14}\\
&\lesssim \Big(\sum_{k,l\in\mathbb{Z}_+^{d/2}}(|k|+1)^{4m+d}(|l|+1)^{4m+d}|a_{k,l}|^2\Big)^{\frac14} \\
&\lesssim r_{2m+\frac{d}{2}}(a)^{\frac12}.
\end{aligned}$$
Hence, $r_m(a^{\frac12})^2\lesssim r_{2m+\frac{d}{2}}(a)<\infty,$ for all $m\geq0.$ This finishes the proof.
\end{proof}

We also need a matrix version of Proposition \ref{root lemma}. 
\begin{lem}\label{root lemma matrix case} Let $n\in \mathbb{N}.$ If $0\leq X\in M_n(\mathbb{C})\otimes\mathcal{S}(\qr),$ then $X^{\frac12}\in M_n(\mathbb{C})\otimes\mathcal{S}(\qr).$
\end{lem}
We shall split it into several lemmas. In the following, we say a matrix in $\ell_2(\mathbb{Z}_+^{d/2}\times\mathbb{Z}_+^{d/2})$ is positive if it acts on the Hilbert space $\ell_2(\mathbb{Z}_+^{d/2})$ as a positive operator. 
\begin{lem}\label{rapidly decreasing summand lemma} If $\{Y_j\}_{j=1}^N\subset\ell_2(\mathbb{Z}_+^{d/2}\times\mathbb{Z}_+^{d/2})$ are positive matrices  such that $Y_1+\cdots+Y_N\in\mathbf{S}_d,$ then $Y_1,\cdots,Y_N\in\mathbf{S}_d.$
\end{lem}
\begin{proof} Let $Y_j=(y_{k,l}^j)_{k,l\in \mathbb{Z}_+^{\frac d2}},$ $j=1,\cdots,N.$  Let $\{e_k\}_{k\in\mathbb{Z}_+^{d/2}}\subset \ell_2(\mathbb{Z}_+^{d/2})$ be standard basis vectors. Fix $1\leq j\leq N$ and let $k,l\in\mathbb{Z}_+^{d/2}.$ By positivity of $Y_j$ and by the Cauchy inequality,
$$\begin{aligned}
|y^j_{k,l}|&=|\langle Y_je_k, e_l\rangle|=|\langle Y_j^{\frac12}e_k, Y_j^{\frac12}e_l\rangle|\\
&\leq\|Y_j^{\frac12}e_k\|_{\ell_2}\|Y_j^{\frac12}e_l\|_{\ell_2}=\langle Y_je_k, e_k\rangle^{\frac12}\cdot\langle Y_je_l, e_l\rangle^{\frac12}= (y^j_{k,k})^{\frac12}\cdot (y^j_{l,l})^{\frac12}. 
\end{aligned}$$ 
If we set $\lambda:=\sum_{j=1}^NY_j,$ then by assumption $\lambda\in{\bf S}_d,$ and
$$|y_{k,l}^j|^2\leq\big(\sum_{i=1}^N y_{k,k}^i\big)\big(\sum_{i=1}^Ny_{l,l}^i\big)=\lambda_{k,k}\lambda_{l,l},\quad k,l\in\mathbb{Z}_+^{\frac d2}.$$
It follows from the H\"older inequality that
$$r_m(Y_j)^2\leq \Big(\sum_{k\in\mathbb{Z}_+^{d/2}}(|k|+1)^{2m}\lambda_{k,k}\Big)^2\lesssim \sum_{k\in\mathbb{Z}_+^{d/2}}(|k|+1)^{4m+4d}\lambda^2_{k,k}.$$
Therefore, $r_m(Y_j)^2\lesssim r_{m+d}(\lambda)^2<\infty,$ for all $m\geq0. $ 
\end{proof}

\begin{lem}\label{bi root lemma} Let $Y\in\ell_2(\mathbb{Z}_+^{\frac d2}\times \mathbb{Z}_+^{\frac d2})$ such that $Y^{\ast}Y,YY^{\ast}\in\mathbf{S}_d.$ We have $Y\in\mathbf{S}_d.$
\end{lem}
\begin{proof} Let $Y=(y_{k,l})_{k,l\in \mathbb{Z}_+^{\frac d2}},$ $Y^{\ast}Y=(a_{k,l})_{k,l\in \mathbb{Z}_+^{\frac d2}}$ and $YY^{\ast}=(b_{k,l})_{k,l\in \mathbb{Z}_+^{\frac d2}}.$ 
Let $\{e_{k,l}\}_{k,l\in\mathbb{Z}_+^{\frac d2}}\subset\ell_2(\mathbb{Z}_+^{\frac d2}\times \mathbb{Z}_+^{\frac d2})$ be the standard matrix units.
We have
$e_{k,k}Ye_{l,l}=y_{k,l}e_{k,l},$ for all $k,l\in \mathbb{Z}_+^{\frac d2}.$
Hence,
$|y_{k,l}|^2=\|e_{k,k}Ye_{l,l}\|_{\mathcal{L}_2}^2\leq\|e_{k,k}Y\|_{\mathcal{L}_2}^2={\rm Tr}(e_{k,k}YY^{\ast}e_{k,k})=b_{k,k},$ similarly,
$|y_{k,l}|^2\leq a_{l,l}.$
Consequently,
$|y_{k,l}|^2\leq a_{l,l}^{\frac12}b_{k,k}^{\frac12}.$
The result follows now from the H\"older inequality.
\end{proof}

\begin{lem}\label{matrix root lemma} If $0\leq X\in M_n(\mathbb{C})\otimes \mathbf{S}_d,$ then also $X^{\frac12}\in M_n(\mathbb{C})\otimes\mathbf{S}_d.$
\end{lem}
\begin{proof} Let $\{E_{i,j}\}_{i,j=1}^n\subset M_n(\mathbb{C})$ be matrix units.
We may write
$$X^{\frac12}=\sum_{i,j=1}^nE_{i,j}\otimes Y_{i,j},\quad X=\sum_{i,j=1}^nE_{i,j}\otimes X_{i,j},\quad X_{i,j}\in\mathbf{S}_d.$$
Since $X^{\frac12}$ is self-adjoint, it follows that
$$X=\big(\sum_{i,j=1}^nE_{i,j}\otimes Y_{i,j}\big)\cdot (\sum_{i,j=1}^n E_{j,i}\otimes Y_{i,j}^{\ast})=\sum_{i_1,i_2=1}^n E_{i_1,i_2}\otimes\big(\sum_{j=1}^nY_{i_1,j}Y_{i_2,j}^{\ast}\big).$$
Thus,
$X_{i_1,i_2}=\sum_{j=1}^nY_{i_1,j}Y_{i_2,j}^{\ast},$ for $1\leq i_1,i_2\leq n.$
In particular,
$X_{i,i}=\sum_{j=1}^nY_{i,j}Y_{i,j}^{\ast},$ for $1\leq i\leq n.$
By Lemma \ref{rapidly decreasing summand lemma}, $Y_{i,j}Y_{i,j}^{\ast}\in\mathbf{S}_d$ for $1\leq i,j\leq n.$ Since $X^{\frac12}$ is self-adjoint, it follows that $Y_{i,j}=Y_{j,i}^{\ast}$ for $1\leq i,j\leq n.$ Thus, $Y_{j,i}^{\ast}Y_{j,i}\in\mathbf{S}_d$ for $1\leq i,j\leq n.$ Swapping $i$ and $j,$ we infer that $Y_{i,j}^{\ast}Y_{i,j}\in\mathbf{S}_d$ for $1\leq i,j\leq n.$ Since also $Y_{i,j}Y_{i,j}^{\ast}\in\mathbf{S}_d$ for $1\leq i,j\leq n,$ it follows from Lemma \ref{bi root lemma} that $Y_{i,j}\in\mathbf{S}_d$ for $1\leq i,j\leq n.$ This completes the proof.
\end{proof}

\begin{proof}[Proof of Lemma \ref{root lemma matrix case}] Due to Corollary \ref{star algebra iso}, we see that 
${\rm id}\otimes U_{\theta}\circ\rho_{\theta}:M_n(\mathbb{C})\otimes {\bf S}_d\to M_n\otimes \mathcal{S}(\qr)$ is an isomorphism between $\ast$-algebras.	
The result now follows from Lemma \ref{matrix root lemma}.
\end{proof}
Another feature of $\mathcal{S}(\qr)$ is factorization property \cite[Propostion 2.7]{Gay2004}:
\begin{prop}\label{factorisable}
For $x\in\mathcal{S}(\qr),$ there are $x_1,x_2\in\mathcal{S}(\qr)$ such that $x=x_1x_2.$ 
\end{prop}

\section{Commutator estimates}\label{sec-commutator}
This section is devoted to prove the following commutator estimates. Throughout we assume that $\theta\in M_d(\mathbb{R})$ is non-degenerate and anti-symmetric. In addition, we use the convention $\mathcal{L}_{\infty,\infty}=\mathcal{L}_{\infty}.$ 
\begin{thm}\label{commutator-PSD}
Let $\alpha,\beta \in \mathbb{R},$ and let $x\in \mathcal{S}(\qr),$ $g\in C^{\infty}(\mathbb{S}^{d-1}).$ If $\alpha\leq\beta+1,$ then
$$[\pi_1(x),\pi_2(g)(1-\Delta_{\theta})^{\frac{\alpha}{2}}](1-\Delta_{\theta})^{-\frac{\beta}{2}}\in \mathcal{L}_{\frac{d}{\beta-\alpha +1},\infty}.$$
\end{thm}
When $g=1,$ Theorem \ref{commutator-PSD} recovers an earlier result of Mcdonald-Sukochev-Xiao \cite[Theorem 1.6]{MSX2020}. That is,
\begin{equation}\label{commutator without g}
[\pi_1(x),(1-\Delta_{\theta})^{\frac{\alpha}{2}}](1-\Delta_{\theta})^{-\frac{\beta}{2}}\in \mathcal{L}_{\frac{d}{\beta-\alpha +1},\infty},\quad x\in\mathcal{S}(\qr),\quad \alpha\leq \beta+1.
\end{equation} 

The proving of Theorem \ref{commutator-PSD} will be accomplished by a series of lemmas. The following priori result \cite[Lemma 5.4]{MSX2020} is an essential ingredient.

\begin{lem}\label{arbitrary index for Bessel}
Let $x\in\mathcal{S}(\qr).$ Then for all $\beta>0$ we have
$$\pi_1(x)(1-\Delta_{\theta})^{-\frac{\beta}{2}},~[\pi_1(x),(1-\Delta_{\theta})^{\frac12}](1-\Delta_{\theta})^{-\frac{\beta}{2}}\in\mathcal{L}_{\frac{d}{\beta},\infty}.$$
\end{lem}
As an application of Theorem \ref{commutator-PSD}, we obtain the following result.
\begin{cor}\label{coro1}
Let $x\in \mathcal{S}(\qr)$ and $g\in C(\mathbb{S}^{d-1}).$ We have
$$[\pi_1(x),\pi_2(g)](1-\Delta_{\theta})^{-\frac12}, [\pi_1(x),\pi_2(g)(1-\Delta_{\theta})^{-\frac12}]\in(\mathcal{L}_{d,\infty})_0.$$
\end{cor}
\begin{proof}
We write
$$[\pi_1(x),\pi_2(g)(1-\Delta_{\theta})^{-\frac12}]=[\pi_1(x),\pi_2(g)](1-\Delta_{\theta})^{-\frac12}+\pi_2(g)[\pi_1(x),(1-\Delta_{\theta})^{-\frac12}].$$ 
By taking $\alpha=-1$ and $\beta=0$ in \eqref{commutator without g}, the second summand is in $\mathcal{L}_{\frac{d}{2},\infty}\subset(\mathcal{L}_{d,\infty})_0.$ So we only need to deal with the first summand. 
To this end, let $x$ be fixed. Note that if $g\in C^{\infty}(\mathbb{S}^{d-1}),$
then the first assertion follows immediately by taking $\alpha=0,\beta=1$ in Theorem \ref{commutator-PSD}. Furthermore, note that $C^\infty(\mathbb{S}^{d-1})$ is dense in $C(\mathbb{S}^{d-1}),$ and $(\mathcal{L}_{d,\infty})_0$ is closed in $\mathcal{L}_{d,\infty}.$ Hence, it suffices to show the continuity of the following mapping from $C(\mathbb{S}^{d-1})$ to $\mathcal{L}_{d,\infty}:$
$g\mapsto[\pi_1(x),\pi_2(g)](1-\Delta_{\theta})^{-\frac12}.$
We write 
$[\pi_1(x),\pi_2(g)](1-\Delta_{\theta})^{-\frac12}=[\pi_1(x)(1-\Delta_{\theta})^{-\frac12},\pi_2(g)].$
Thus, using the quasi-triangle inequality, we obtain
$$
\|[\pi_1(x),\pi_2(g)](1-\Delta_{\theta})^{-\frac12}\|_{d,\infty}
\leq 2^{1+\frac1d}\|\pi_2(g)\|_{\infty} \|\pi_1(x)(1-\Delta_{\theta})^{-\frac12}\|_{d,\infty}.
$$
Clearly, $\pi_2(g)\leq\|g\|_\infty,$ and by Lemma \ref{arbitrary index for Bessel} the second factor is finite. This implies the desired continuity and therefore, completes the proof. 
\end{proof}

\begin{defi}\rm{}
With a slight abuse of notation, we denote by $\nabla_{\theta}$ the gradient operator associated with $L_\infty(\qr)$:
$$
\nabla_{\theta}:=\left(D_1, D_2,\ldots,D_d\right):\bigcap_{k=1}^d\mathrm{dom}(D_k)\to L_2(\mathbb{R}^d)^d.
$$
For any $g\in L_{\infty}(\mathbb{R}^d),$ let $g(\nabla_{\theta})$ be the multiplication operator on $L_2(\mathbb{R}^d)$
defined by
$
(g(\nabla_{\theta})\xi)(t):=g(t)\xi(t),$ $t\in\mathbb{R}^d.
$
Clearly, $g(\nabla_\theta)$ is bounded on $L_2(\mathbb{R}^d).$

For $1\leq k\leq d,$ set $h_k(t)=\frac{t_k}{|t|},$ $0\neq t\in\mathbb{R}^d.$ In particular, the operator  $R_k:=h_k(\nabla_{\theta})=D_k(-\Delta_{\theta})^{-\frac12}$ is called the $k$-th Riesz transform associated with $L_\infty(\qr).$   
\end{defi}

\begin{lem}\label{commutator-PSD alfa0 small beta rk}
Let $\beta\geq-1$ and $x\in \mathcal{S}(\qr).$ We have
$$[\pi_1(x),R_k](1-\Delta_{\theta})^{-\frac{\beta}{2}}\in\mathcal{L}_{\frac{d}{\beta+1},\infty},\quad 1\leq k\leq d.$$
\end{lem}
\begin{proof}

Fix $k$ and set
$$g_k(t)=\frac{t_k}{|t|}-\frac{t_k}{(1+|t|^2)^{\frac12}},\quad 0\neq t=(t_1,\cdots,t_d)\in\mathbb{R}^d.$$
It is clear that $R_k=g_k(\nabla_{\theta})+D_k(1-\Delta_{\theta})^{-\frac12}.$ We have
$$
\begin{aligned}
[\pi_1(x),R_k](1-\Delta_{\theta})^{-\frac{\beta}{2}}&=[\pi_1(x),g_k(\nabla_{\theta})](1-\Delta_{\theta})^{-\frac{\beta}{2}}\\&+[\pi_1(x),D_k(1-\Delta_{\theta})^{-\frac12}](1-\Delta_{\theta})^{-\frac{\beta}{2}}
\overset{def}{=}I_{\beta}+II_{\beta}.
\end{aligned}
$$
Let us treat these two summand separately. First, we split $I_\beta$ into:
\small{$$I_{\beta}=[\pi_1(x)(1-\Delta_{\theta})^{-\frac{\beta+1}{2}},(1-\Delta_{\theta})^{\frac12}g_k(\nabla_{\theta})]+g_{k}(\nabla_{\theta})\cdot[(1-\Delta_{\theta})^{\frac12},\pi_1(x)](1-\Delta_{\theta})^{-\frac{\beta+1}{2}}.
$$}
By construction, we have $g_k(t)=O((1+|t|^2)^{-\frac12}).$
So $g_k(\nabla_{\theta})$ and $(1-\Delta_{\theta})^{\frac12}g_k(\nabla_{\theta})$ are both bounded operators. 
It then follows from Lemma \ref{arbitrary index for Bessel} that for $\beta>-1,$ $$I_{\beta}\in[\mathcal{L}_{\frac{d}{\beta+1},\infty},\mathcal{L}_{\infty}]+\mathcal{L}_{\infty}\cdot\mathcal{L}_{\frac{d}{\beta+1},\infty}\subset\mathcal{L}_{\frac{d}{\beta+1},\infty}.$$
For $\beta=-1,$ note that
$
I_{-1}=[\pi_1(x),(1-\Delta_{\theta})^{\frac12}g_k(\nabla_{\theta})]+g_{k}(\nabla_{\theta})\cdot[(1-\Delta_{\theta})^{\frac12},\pi_1(x)].
$
By \eqref{commutator without g}, $ [(1-\Delta_{\theta})^{\frac12},x]$ has a bounded extension in $\mathcal{L_\infty}.$ Thus, $I_{-1}\in\mathcal{L}_{\infty}.$ So far, we  have proved that $I_\beta\in\mathcal{L}_{\frac{d}{\beta+1},\infty}$ for all $\beta\geq-1.$ It remains to handle the second summand $II_{\beta}.$
We write
$$
\begin{aligned}
II_{\beta}&=[\pi_1(x),D_k(1-\Delta_{\theta})^{-\frac12}](1-\Delta_{\theta})^{-\frac{\beta}{2}}\\
&=[\pi_1(x),D_k](1-\Delta_{\theta})^{-\frac{\beta+1}{2}}+D_k[\pi_1(x), (1-\Delta_{\theta})^{-\frac12}](1-\Delta_{\theta})^{-\frac{\beta}{2}}\\
&=-\pi_1(\partial_kx) \cdot(1-\Delta_{\theta})^{-\frac{\beta+1}{2}}+D_k(1-\Delta_{\theta})^{-\frac12}\cdot[(1-\Delta_{\theta})^{\frac12},\pi_1(x)](1-\Delta_{\theta})^{-\frac{\beta+1}{2}}.
\end{aligned}
$$
By assumption, $\partial_kx\in\mathcal{S}(\qr).$ By definition, the multiplication operator $D_k(1-\Delta_{\theta})^{-\frac12}$ induced by the function $t\mapsto t_k(1+|t|^2)^{-\frac12}$ is clearly bounded. Now, similarly, using Lemma \ref{arbitrary index for Bessel} for $\beta>-1,$ and the fact $[(1-\Delta_{\theta})^{\frac12},\pi_1(x)]\in\mathcal{L}_\infty$ for $\beta=-1,$ we conclude that
$II_{\beta}\in\mathcal{L}_{\frac{d}{\beta+1},\infty}$ for all $\beta\geq-1.$ The proof is therefore complete.
\end{proof}

\begin{lem}\label{commutator-PSD alfa0 small beta}
Let $-1\leq\beta<1$ and let $x\in \mathcal{S}(\qr),$ $g\in C^{\infty}(\mathbb{S}^{d-1}).$ We have
$$[\pi_1(x),\pi_2(g)](1-\Delta_{\theta})^{-\frac{\beta}{2}}\in \mathcal{L}_{\frac{d}{\beta+1},\infty}.$$
\end{lem}
\begin{proof}Denote $p=\frac{d}{\beta+1}.$ Clearly, $p>1.$ For $g\in C^{\infty}(\mathbb{S}^{d-1}),$ choose $h\in\mathcal{S}(\mathbb{R}^d)$ such that $h|_{\mathbb{S}^{d-1}}=g$ (see e.g. \cite[Lemma~4.8]{FSZ2023}). Note that by the functional calculus, 
$\pi_{2}(g)=h(R_1,\cdots,R_d).$ Thus by means of \cite[Lemma 3.5]{SXZ2023} (similar assertion holds trivially for the case $p=\infty$), there exists a constant $c_{p,h}>0$ such that 
\begin{equation*}
\begin{split}
	\|[\pi_1(x),\pi_2(g)](1-\Delta_{\theta})^{-\frac{\beta}{2}}\|_{p,\infty}&=\|[\pi_1(x)(1-\Delta_{\theta})^{-\frac{\beta}{2}},\pi_2(g)]\|_{p,\infty}\\
	&=\|[\pi_1(x)(1-\Delta_{\theta})^{-\frac{\beta}{2}},h(R_1,\cdots,R_d)]\|_{p,\infty}\\
	&\leq c_{p,h}\max_{1\leq k\leq d }\|[\pi_1(x)(1-\Delta_{\theta})^{-\frac{\beta}{2}},R_k]\|_{p,\infty}\\
	&=c_{p,h}\max_{1\leq k\leq d }\|[\pi_1(x),R_k](1-\Delta_{\theta})^{-\frac{\beta}{2}}\|_{p,\infty}.
\end{split}
\end{equation*}
The result now follows from Lemma \ref{commutator-PSD alfa0 small beta rk}.
\end{proof}

Next, we extend the range of $\beta$ from $[-1,1)$ to $[-1,\infty).$
\begin{lem}\label{commutator-PSD alfa0}
Let $\beta\geq-1,$ and let $x\in \mathcal{S}(\qr),$ $g\in C^{\infty}(\mathbb{S}^{d-1}).$ We have
$$[\pi_1(x),\pi_2(g)](1-\Delta_{\theta})^{-\frac{\beta}{2}}\in \mathcal{L}_{\frac{d}{\beta+1},\infty}.$$
\end{lem}
\begin{proof}Let $n=\lfloor 2\beta\rfloor.$ We proceed the proof by induction on $n.$ Base of induction (i.e., $n=1$ or, equivalently, $\beta\in [\frac12,1)$) is established in Lemma \ref{commutator-PSD alfa0 small beta}. We now concentrate on the step of induction.

Suppose now the result holds for $n,$ we then need to consider the case $\beta\in [\frac{n+1}{2},\frac{n}{2}+1).$ Setting $\alpha=\beta-\frac12,$ it is obvious that $n=\lfloor 2\alpha\rfloor.$
Noting that $\pi_2(g)$ commutes with $(1-\Delta_{\theta})^{-1},$ we may write
$$
\begin{aligned}
[\pi_1(x),\pi_2(g)](1-\Delta_{\theta})^{-\frac{\beta}{2}}&=[\pi_1(x),\pi_2(g)](1-\Delta_{\theta})^{-\frac{\alpha+\frac12}{2}}\\
&=[[\pi_1(x),(1-\Delta_{\theta})^{-\frac14}](1-\Delta_{\theta})^{-\frac{\alpha}{2}},\pi_2(g)]+\\
&(1-\Delta_{\theta})^{-\frac14}[\pi_1(x),\pi_2(g)](1-\Delta_{\theta})^{-\frac{\alpha}{2}}.
\end{aligned}
$$
Since $\pi_2(g)$ is bounded, then by \eqref{commutator without g} the first summand belongs to $\mathcal{L}_{\frac{d}{\beta+1},\infty}.$ 

It remains to deal with the second summand, which we denote by $II.$ Recall that by Proposition \ref{factorisable} we may write $x=x_1x_2,$ with $x_1,x_2\in\mathcal{S}(\qr).$
Thus,
$$\begin{aligned}
II&=(1-\Delta_{\theta})^{-\frac14}[\pi_1(x_1),\pi_2(g)]\cdot \pi_1(x_2)(1-\Delta_{\theta})^{-\frac{\alpha}{2}}\\
&+(1-\Delta_{\theta})^{-\frac14}\pi_1(x_1)\cdot [\pi_1(x_2),\pi_2(g)](1-\Delta_{\theta})^{-\frac{\alpha}{2}}.
\end{aligned}
$$
Using Lemma \ref{commutator-PSD alfa0 small beta}, Lemma \ref{arbitrary index for Bessel} and the inductive assumption, we obtain
$$II\in\mathcal{L}_{\frac{2d}{3},\infty}\cdot\mathcal{L}_{\frac{d}{\alpha},\infty}+\mathcal{L}_{2d,\infty}\cdot \mathcal{L}_{\frac{d}{\alpha+1},\infty}\subset \mathcal{L}_{\frac{d}{\beta+1},\infty}.$$
This establishes the step of induction and hence completes the proof.
\end{proof}

We are now in position to prove our main result on commutator estimates.

\begin{proof}[Proof of Theorem \ref{commutator-PSD}] Recall by assumption $\alpha\leq \beta+1.$ We write
\begin{equation*}
\begin{split}
	[\pi_1(x),\pi_2(g)(1-\Delta_{\theta})^{\frac{\alpha}{2}}](1-\Delta_{\theta})^{-\frac{\beta}{2}}
	&=[\pi_1(x),\pi_2(g)](1-\Delta_{\theta})^{\frac{\alpha-\beta}{2}}\\&+\pi_2(g)\cdot[\pi_1(x), (1-\Delta_{\theta})^{\frac{\alpha}{2}}](1-\Delta_{\theta})^{-\frac{\beta}{2}}
\end{split}
\end{equation*}
The result now follows from Lemma \ref{commutator-PSD alfa0} and \eqref{commutator without g}.
\end{proof}

\section{Mixed trace formula}\label{sec-mixed trace formula}
In this section, we establish a trace formula for convolution-product type operators of the form $\pi_1(x)g(\nabla_{\theta}).$ Under certain conditions, say $x=U_\theta(f)$ for some $f\in L_2(\mathbb{R}^d)$ and $g\in L_2(\mathbb{
R}^d),$ $\pi_1(x)g(\nabla_{\theta})$ turns out to be an integral operator on $\mathbb{R}^d$ with the kernel $(t,s)\mapsto 
e^{-\frac{i}{2}\langle t,\theta s\rangle}f(t-s)g(s).$

Recall the space $\ell_1(L_\infty)(\mathbb{R}^d)$  is defined as the set of functions $g\in L_\infty(\mathbb{R}^d)$ such that
$$\|g\|_{\ell_1(L_\infty)(\mathbb{R}^d)}:=\sum_{\alpha\in\mathbb{Z}^d}\left\|g\chi_{[0,1)^d+\alpha}\right\|_{\infty}<\infty.$$
For every $\beta>d,$ one can easily check that the function
$t\to(1+|t|^2)^{-\beta/2}$
belongs to $\ell_1(L_{\infty})(\mathbb{R}^d).$

In \cite{LSZ2020}, the following Cwikel estimates were established (see Lemma 7.1 and Theorem 7.7 there). 
\begin{thm}\label{cwikel-type estimates}Assume $\det(\theta)\neq0.$ 
\begin{enumerate}[(i)]
\item\label{H-S norm} If $x\in L_2(\qr)$ and if $g\in L_2(\mathbb{R}^d),$ then $\pi_1(x)g(\nabla_{\theta})\in\mathcal{L}_2(L_2(\mathbb{R}^d))$ and
$$\|\pi_1(x)g(\nabla_{\theta})\|_{\mathcal{L}_2(L_2(\mathbb{R}^d))}=(2\pi)^{-\frac{d}{2}}\|x\|_{L_2(\qr)}\|g\|_{L_2(\mathbb{R}^d)}.$$
\item\label{Sobolev norm} If $x \in W_1^{d}(\mathbb{R}_{\theta}^{d})$ and if $g\in$ $\ell_{1}(L_{\infty})(\mathbb{R}^{d}),$ then $\pi_1(x)g(\nabla_{\theta})\in \mathcal{L}_1(L_{2}(\mathbb{R}^{d}))$ and
$$\left\|\pi_1(x) g\left( \nabla_{\theta}\right)\right\|_{\mathcal{L}_1\left(L_{2}\left(\mathbb{R}^{d}\right)\right)} \lesssim\|x\|_{W_1^{d}\left(\mathbb{R}_{\theta}^{d}\right)}\|g\|_{\ell_{1}\left(L_{\infty}\right)(\mathbb{R}^d)}.$$
\end{enumerate}
\end{thm}

\begin{lem}(\cite[Lemma 3.4]{SZ2018-cmp})\label{sftinv}
If $\phi$ is a continuous linear functional on $W_{1}^{d}\left(\mathbb{R}_{\theta}^{d}\right)$ such that
\[\phi(x)=\phi(U_\theta(-t)xU_\theta(t)), \quad x\in W_{1}^{d}(\mathbb{R}_{\theta}^{d}), \quad t \in\mathbb{R}^{d},\]
then $\phi=\tau_{\theta}$ (up to a constant factor).
\end{lem}

\begin{lem}\label{single trace lemma}
For every $g\in \ell_1 (L_\infty )(\mathbb{R}^d),$ there exists a constant $c_g\in\mathbb{C}$ depending on $g$ such that
\begin{equation}\label{mtt eq0}
\mathrm{Tr}(\pi_1(x)g(\nabla_{\theta}))=c_g\tau_{\theta}(x),\quad x\in W_1^d(\mathbb{R}_\theta^d).
\end{equation} 
In particular, $g\mapsto c_g$ is a linear functional.
\end{lem}

\begin{proof}
Fix $g\in\ell_1(L_\infty)(\mathbb{R}^d),$ and set $$\phi_g(x)=\mathrm{Tr}(\pi_1(x)g(\nabla_{\theta})),\quad x\in W^d_1(\qr).$$ 
By Theorem \ref{cwikel-type estimates} \eqref{Sobolev norm}, $\phi_g$ is a continuous linear functional on $W^d_1(\qr).$

Next, fix $t\in\mathbb{R}^d$ and consider the following two unitary operators defined by
$$ (e^{\pm {\rm i}\langle\nabla_{\theta},\theta t\rangle}\xi)(u):=e^{\pm {\rm i}\langle u,\theta t\rangle}\xi(u),\quad\xi\in L_2(\mathbb{R}^d).$$
By the unitary invariance of $\mathrm{Tr},$ we write
$$\phi_g(x)=\mathrm{Tr}\big(e^{{\rm i}\langle\nabla_{\theta},\theta t\rangle}\pi_1(x)g(\nabla_{\theta})e^{-{\rm i}\langle\nabla_{\theta},\theta t\rangle}\big),$$
Noting that $g(\nabla_{\theta})$ commutes with $e^{-{\rm i}\langle\nabla_{\theta},\theta t\rangle},$ it follows that
$$\phi_g(x)=\mathrm{Tr}\big(e^{{\rm i}\langle\nabla_{\theta},\theta t\rangle}\pi_1(x)e^{-{\rm i}\langle\nabla_{\theta},\theta t\rangle}\cdot g(\nabla_{\theta})\big).$$
We now claim that 
\begin{equation*}
e^{{\rm i}\langle\nabla_{\theta},\theta t\rangle}\pi_1(x)e^{-{\rm i}\langle\nabla_{\theta},\theta t\rangle}=\pi_1\big(U_\theta(-t)xU_\theta(t)\big),\quad x\in L_\infty(\qr).
\end{equation*}
Since $\{U_\theta(s)\}_{s\in\mathbb{R}^d}$ generates the algebra $L_\infty(\qr),$ it suffices to show this claim for $x=U_{\theta}(s),$ $s\in\mathbb{R}^d.$  
Indeed, for every $\xi\in L_2(\mathbb{R}^d)$ we have
$$\begin{aligned}
\big(e^{{\rm i}\langle\nabla_{\theta},\theta t\rangle} U_\theta(s) e^{-{\rm i}\langle\nabla_{\theta},\theta t\rangle}\xi\big)(u)&=e^{{\rm i}\langle u,\theta t\rangle}\cdot \big(U_\theta(s) e^{-{\rm i}\langle\nabla_{\theta},\theta t\rangle}\xi\big)(u)\\&=e^{{\rm i}\langle u,\theta t\rangle}\cdot e^{-\frac{\rm i}{2}\langle s,\theta u\rangle}\cdot(e^{-{\rm i}\langle\nabla_{\theta},\theta t\rangle}\xi)(u-s)\\&=e^{{\rm i}\langle u,\theta t\rangle}\cdot e^{-\frac{\rm i}{2}\langle s,\theta u\rangle}\cdot e^{-{\rm i}\langle u-s,\theta t\rangle}\cdot\xi(u-s)\\&=e^{{\rm i}\langle s,\theta t\rangle}\cdot (U_{\theta}(s)\xi)(u)\\&\overset{\eqref{weyl relation}}{=}(U_\theta(-t) U_\theta(s) U_\theta(t)\xi)(u).
\end{aligned}
$$
This yields the claim. 
Combining the preceding paragraphs, we arrive at $$\phi_g(x)=\mathrm{Tr}\big(\pi_1(U_\theta(-t)xU_\theta(t))\cdot g(\nabla_{\theta})\big)=\phi_g(U_\theta(-t)xU_\theta(t)),\quad x\in W^d_1(\qr).$$
The result now follows from Lemma \ref{sftinv}.
\end{proof}

We need the following translation invariance property \cite[Lemma 7.5]{LSZ2020}.
\begin{lem}\label{Vt}
For every $t\in\mathbb{R}^d,$ there exists a unitary operator $V(t)$ on $L_{2}(\mathbb{R}^d)$ such that
\begin{enumerate}[(i)]
\item\label{commutativity} $[V(t), x]=0,\quad\forall x\in L_{\infty}(\qr).$
\item \label{translation} $V(t)e^{{\rm i}sD_j}V(t)^\ast=e^{{\rm i}s(D_j+t_j)},\quad s\in\mathbb{R},\quad 1\leq j\leq d.$
\end{enumerate}
\end{lem}

In what follows, we denote by $T_t,$ $t\in\mathbb{R}^d$ the translation operators acting on an arbitrary function $\xi$ on $\mathbb{R}^d$ by the formula $(T_t\xi)(u):=\xi(u+t),$ $u\in\mathbb{R}^d.$
The Fourier transform $\hat{f}$ of an integrable function $f\in L_1(\mathbb{R}^d)$ is defined by
$$\hat{f}(t):=\frac{1}{(2\pi)^d}\int_{\mathbb{R}^d}f(s)e^{-{\rm i}\langle t,s\rangle}ds.$$
In this setting, we have
$e^{{\rm i}\langle t,s \rangle}\hat{f}(s)=\widehat{T_tf}(s),$
and the corresponding Fourier inversion now becomes:
$$f(t)=\int_{\mathbb{R}^d}\hat{f}(s)e^{{\rm i}\langle t,s\rangle}ds,\quad \hat{f}\in L_1(\mathbb{R}^d).$$

\begin{lem}\label{translation invariant lemma}
For any $g\in \mathcal{S}(\mathbb{R}^d)$ and $t\in\mathbb{R},$ we have
$V(t)g(\nabla_{\theta})V(t)^{\ast}=(T_tg)(\nabla_{\theta}).$
\end{lem}
\begin{proof}
By Fourier inversion and functional calculus, we may write
$$g(\nabla_{\theta})=\int_{\mathbb{R}^d}\hat{g}(s)e^{{\rm i}\langle \nabla_{\theta},s\rangle}ds.$$
Thus, it follows from Lemma \ref{Vt} \eqref{translation} that
$$
\begin{aligned}
V(t)g(\nabla_{\theta})V(t)^{\ast}&=\int_{\mathbb{R}^d}\widehat{g}(s)V(t)e^{{\rm i}\langle \nabla_{\theta},s\rangle}V(t)^\ast ds\\
&=\int_{\mathbb{R}^d}\widehat{g}(s)e^{{\rm i}\langle \nabla_{\theta}+t,s\rangle}ds\\&=\int_{\mathbb{R}^d}\widehat{T_tg}(s)e^{{\rm i}\langle \nabla_{\theta},s\rangle}ds=(T_tg)(\nabla_{\theta}).
\end{aligned}$$
\end{proof}

The following result should be compared with Theorem 1.1 in \cite{SZ2018-cmp}, where a similar trace formula was established in terms of singular traces on $\mathcal{L}_{1,\infty}.$
\begin{thm}\label{mixed trace thm}
For $x\in W_1^d(\mathbb{R}_\theta^d)$ and $g\in \ell_1 (L_\infty)(\mathbb{R}^d),$ we have
$$\mathrm{Tr}(\pi_1(x)g(\nabla_{\theta}))=\frac{1}{(2\pi)^d}\tau_{\theta}(x)\int_{\mathbb{R}^d} g.$$
\end{thm}

\begin{proof}
Fix $x\in W_1^d(\qr)$ and
consider the functional
$$l_x(g):=\mathrm{Tr}(\pi_1(x^{\ast}x)g(\nabla_{\theta})),\quad g\in\ell_1 (L_\infty)(\mathbb{R}^d).$$
By Theorem \ref{cwikel-type estimates}\eqref{Sobolev norm}, $\pi_1(x)g(\nabla_{\theta})\in\mathcal{L}_1,$ and $l_x$ is a well-defined continuous linear functional. By the tracial property, we may write
$l_x(g)=\mathrm{Tr}(\pi_1(x)g(\nabla_{\theta})\pi_1(x)^{\ast}).$
Using the identity
$$\pi_1(x)g(\nabla)\pi_1(x)^{\ast}=\pi_1(x)|g|^{\frac12}(\nabla_{\theta})\cdot\frac{g}{|g|}(\nabla_{\theta})\cdot\big(\pi_1(x)|g|^{\frac12}(\nabla_{\theta})\big)^{\ast}$$
and the Cauchy-Schwartz inequality,
we obtain
$ |l_x(g)|\leq\|\pi_1(x)|g|^{\frac12}(\nabla_{\theta})\|_{\mathcal{L}_2\left(L_{2}\left(\mathbb{R}^{d}\right)\right)}^2.$
Thus, by Theorem \ref{cwikel-type estimates}\eqref{H-S norm}, we have
$$|l_x(g)|\lesssim\|x\|_{L_2(\qr)}^2\||g|^{\frac12}\|_{L_2(\mathbb{R}^d)}^2\lesssim\|x\|_{L_2(\qr)}^2\|g\|_{L_1(\mathbb{R}^d)}.$$
Hence, $l_x$ extends to a continuous linear functional on $L_1(\mathbb{R}^d)$.

For a Schwartz function $g\in\mathcal{S}(\mathbb{R}^d),$ by the unitary invariance of $\mathrm{Tr},$ Lemma \ref{Vt} \eqref{commutativity} and Lemma \ref{translation invariant lemma}, we have
$$\begin{aligned}
l_x(g)&=\mathrm{Tr}\big(V(t)\cdot \pi_1(x^\ast x)g(\nabla_{\theta})\cdot V(t)^{\ast}\big)\\&=\mathrm{Tr}\big(\pi_1(x^\ast x)\cdot V(t) g(\nabla_{\theta})V(t)^{\ast}\big)\\&=\mathrm{Tr}(\pi_1(x^\ast x)(T_tg)(\nabla_{\theta}))
\end{aligned}.$$
Therefore,
$l_x(g)=l_x(T_tg),$ $ g\in\mathcal{S}(\mathbb{R}^d).$
Note that $T_t:L_1(\mathbb{R}^d)\to L_1(\mathbb{R}^d)$ is an isometry, and $\mathcal{S}(\mathbb{R}^d)$ is dense in $L_1(\mathbb{R}^d).$ Thus,
$l_x(g)=l_x(T_tg),$ for all $g\in L_1(\mathbb{R}^d).$
Since $t\in\mathbb{R}^d$ can be arbitrary, it follows that $l_x$ is a translation-invariant continuous linear functional on $L_1(\mathbb{R}^d).$   
Then there exists a constant $c_x\in\mathbb{C}$ such that
$$l_x(g)=c_x\int_{\mathbb{R}^d} g,\quad g\in L_1(\mathbb{R}^d).$$
In particular,
\begin{equation}\label{mtt eq1}
\mathrm{Tr}(\pi_1(x^{\ast}x)g(\nabla_{\theta}))=c_x\int_{\mathbb{R}^d} g,\quad g\in \ell_1 (L_\infty )(\mathbb{R}^d).
\end{equation}	
Furthermore, it follows from Lemma \ref{single trace lemma} that
\begin{equation}\label{mtt eq2}
\mathrm{Tr}(\pi_1(x^{\ast}x)g(\nabla_{\theta}))=c_g\tau_{\theta}(x^{\ast}x),\quad x\in W_1^d(\qr).
\end{equation}
Comparing \eqref{mtt eq1} with \eqref{mtt eq2}, we deduce that
$$c_g=c\int_{\mathbb{R}^d} g,\quad g\in \ell_1 (L_\infty)(\mathbb{R}^d),$$
for some constant $c\in\mathbb{C}.$
Substituting this back to \eqref{mtt eq0}, we obtain
\begin{equation}\label{final constant}
\mathrm{Tr}(\pi_1(x)g(\nabla_{\theta}))=c\tau_\theta(x)\int_{\mathbb{R}^d} g,\quad x\in W^d_1(\qr),\quad g\in \ell_1(L_\infty)(\mathbb{R}^d).
\end{equation}

To determine the constant $c,$ let $f,g\in\mathcal{S}(\mathbb{R}^d)$ be such that $f(0)=(2\pi)^{-d}$ and $\int g=1.$ In this case, $\tau_{\theta}(U_\theta(f))=1.$  
It turns out that $\pi_1(U_{\theta}(f))g(\nabla_{\theta})\in\mathcal{L}_1$ is an integral operator on $L_2(\mathbb{R}^d)$ with the Schwartz class kernel
$ (t,s)\mapsto e^{-\frac{\rm i}{2}\langle t,\theta s\rangle}f(t-s)g(s).$
By Corollary 3.2 in \cite{Brislawn1991}, the classical trace of $\pi_1(U_{\theta}(f))g(\nabla_{\theta})$ is precisely the integral of its kernel on the diagonal, namely,
$$\mathrm{Tr}(\pi_1(U_{\theta}(f))g(\nabla_{\theta}))=\int_{\mathbb{R}^d}f(s-s)g(s)e^{-\frac{\rm i}{2}\langle s,\theta s\rangle}ds=f(0)\cdot\int_{\mathbb{R}^d} g=(2\pi)^{-d}.$$
Hence, from \eqref{final constant} we conclude that
$c=(2\pi)^{-d}.$ This completes the proof.
\end{proof}

The following version of
Theorem \ref{mixed trace thm} concerning matrix tensor factors is also useful.
\begin{cor}\label{mixed trace thm-matrix case}
Let $X\in M_n(\mathbb{C})\otimes W_1^d(\mathbb{R}_\theta^d)$ and $g\in \ell_1 (L_\infty )(\mathbb{R}^d).$ We have
$$(\mathrm{tr}\otimes \mathrm{Tr})\Big(({\rm id}\otimes \pi_1)(X)\cdot(\mathrm{id}\otimes g(\nabla_{\theta}))\Big)=\frac{1}{(2\pi)^d}(\mathrm{tr}\otimes\tau_{\theta})(X)\cdot\int_{\mathbb{R}^d} g.$$
\end{cor}
\begin{proof}
We write $X=(X_{k,l})_{1\leq k,l\leq n},$ with $X_{k,l}\in W^d_1(\qr).$
A direct computation yields
$(\mathrm{tr}\otimes \mathrm{Tr})\Big(({\rm id}\otimes \pi_1)(X)\cdot(\mathrm{id}\otimes g(\nabla_{\theta}))\Big)=\sum_{k=1}^{n}\mathrm{Tr}(X_{k,k}g(\nabla_{\theta}))$ and $ (\mathrm{tr}\otimes\tau_{\theta})(X)\cdot\int_{\mathbb{R}^d} g=\sum_{k=1}^{n}\tau_{\theta}(X_{k,k})\int_{\mathbb{R}^d} g.$
Applying Theorem \ref{mixed trace thm} finishes the proof.
\end{proof}

\section{An explicit expression for operator zeta function}\label{sec-residue}
In this section, we compute rigorously the operator zeta function ${\rm Tr}(A^zB^z)$ defined for suitable operators $A$ and $B,$ so that its singularities can be determined.

In the next two lemmas, we recall criteria for functions having integrable Fourier transform. For the reader's convenience, we include its proofs.
\begin{lem}\label{first zigmund lemma} If $f\in L_1(\mathbb{R})$ is such that
$\|f-T_hf\|_2^2=O(|h|^{1+\alpha}),~h\in\mathbb{R},$ for some $\alpha>0,$ then $\hat{f}\in L_1(\mathbb{R}).$
\end{lem}
\begin{proof} Fix $m\in\mathbb{Z}_+$ and set $h=2^{-m}.$ By Plancherel’s identity, we have
$$
\frac{1}{2\pi}\|f-T_hf\|_2^2=\|\hat{f}-\widehat{T_hf}\|_2^2
=\int_{\mathbb{R}}|1-e^{{\rm i}ht}|^2|\hat{f}(t)|^2dt\geq$$
$$\geq\int_{2^m}^{2^{m+1}}|1-e^{{\rm i}ht}|^2|\hat{f}(t)|^2dt\geq (2-2\cos(1))\int_{2^m}^{2^{m+1}}|\hat{f}(t)|^2dt.
$$

Thus,
$$\int_{2^m}^{2^{m+1}}|\hat{f}(t)|^2dt=O(2^{-m(1+\alpha)}),\quad m\in\mathbb{Z}_+.$$
By Cauchy inequality,
$\int_{2^m}^{2^{m+1}}|\hat{f}(t)|dt=O(2^{-\frac{m\alpha}{2}}),$ for $m\in\mathbb{Z}_+.$
Hence,
$\int_1^{\infty}|\hat{f}(t)|dt<\infty.$
Similarly,
$\int_{-\infty}^{-1}|\hat{f}(t)|dt<\infty.$
Since $f\in L_1(\mathbb{R}),$ it follows that $\hat{f}\in L_{\infty}(\mathbb{R}).$ In particular,
$\int_{-1}^1|\hat{f}(t)|dt<\infty.$
Combining all these inequalities gives the result.
\end{proof}

\begin{lem}\label{second zigmund lemma} If $f\in L_1(\mathbb{R})$ is absolutely continuous and $f'\in (L_1\cap L_p)(\mathbb{R})$ for some $p>1,$ then $\hat{f}\in L_1(\mathbb{R}).$ 	
\end{lem}	

\begin{proof} Suppose without loss of generality that $h>0.$ We have
$\|T_hf-f\|_2^2\leq \|T_hf-f\|_1\|T_hf-f\|_\infty.$	
As $f$ is absolutely continuous, we have
$T_hf-f=\int_0^hT_sf'ds.$
Hence,
$\|T_hf-f\|_1\leq \int_0^h\|T_sf'\|_1ds=h\|f'\|_1.$
Next,
\small{$|f(t+h)-f(t)|=\big|\int_t^{t+h}f'(s)ds\big|\leq\big(\int_t^{t+h}|f'|^p(s)ds\big)^{\frac1p}\cdot\big(\int_t^{t+h}ds\big)^{1-\frac1p}.$}
Hence,
$\|T_hf-f\|_\infty\leq \|f'\|_ph^{1-\frac1p}.$
This means
$\|T_hf-f\|_2^2\leq h^{2-\frac1p}\|f'\|_1\|f'\|_p.$
The result follows now from Lemma \ref{first zigmund lemma}.
\end{proof}

The next result is essentially a variant of classical Duhamel formula.

\begin{prop}\label{abstract Duhamel formula}
Let $B=B^{\ast}$ be a densely-defined operator on the Hilbert space $H.$ If $C$ is a bounded operator  such that $[B,C]$ has bounded extension on $H,$ then
$[B,e^{{\rm i}C}]$ also has bounded extension, moreover,
$$[B,e^{{\rm i}C}]={\rm i}\int_0^1e^{{\rm i}(1-t)C}[B,C]e^{{\rm i}tC}dt.$$
\end{prop}
\begin{proof}
Define the projections $p_n=\chi_{(-n,n)}(B),$ $n\in\mathbb{N}.$ Define the function
$F_n(t)=p_ne^{{\rm i}(1-t)C}p_nBe^{{\rm i}tC}p_n,$ for $t\in [0,1].$
Clearly, each $F_n$ is a continuously differentiable $\mathcal{L}_{\infty}(H)$-valued function.
By Leibniz rule, we have
$$\begin{aligned}
\frac{dF_n(t)}{dt}&=-{\rm i}p_ne^{{\rm i}(1-t)C}Cp_nBe^{{\rm i}tC}p_n+{\rm i}p_ne^{{\rm i}(1-t)C}p_nBCe^{{\rm i}tC}p_n\\&={\rm i}p_ne^{{\rm i}(1-t)C}[p_nB,C]e^{{\rm i}tC}p_n.
\end{aligned} $$
Thus, 
$$p_n[B,e^{{\rm i}C}]p_n=F_n(1)-F_n(0)=\int_0^1\frac{dF_n(t)}{dt}dt={\rm i}\int_0^1p_ne^{{\rm i}(1-t)C}[p_nB,C]e^{{\rm i}tC}p_ndt.$$
Here, the integral is understood in the Bochner sense.
It is clear that $p_n\rightarrow1$ as $n\rightarrow\infty$ in the strong operator topology. Let $\xi,\eta\in {\rm dom}(B)$ be arbitrary, we then have
$\langle p_ne^{{\rm i}(1-t)C}[p_nB,C]e^{{\rm i}tC}p_n\xi,\eta\rangle\to\langle e^{{\rm i}(1-t)C}[B,C]e^{{\rm i}tC}\xi,\eta\rangle,$ as $n\to\infty.$
Therefore,
$$\langle p_n[B,e^{{\rm i}C}]p_n\xi,\eta\rangle\to\langle {\rm i}\int_0^1e^{{\rm i}(1-t)C}[B,C]e^{{\rm i}tC}dt\cdot \xi,\eta\rangle,\quad n\to\infty. $$
Note also that $\langle p_n[B,e^{{\rm i}C}]p_n\xi,\eta\rangle\to\langle[B,e^{{\rm i}C}]\xi,\eta\rangle,$ as $n\to\infty.$
We arrive at
$$\langle[B,e^{{\rm i}C}]\xi,\eta\rangle=\langle {\rm i}\int_0^1e^{{\rm i}(1-t)C}[B,C]e^{{\rm i}tC}dt\cdot \xi,\eta\rangle,\quad \xi,\eta\in{\rm dom}(B).$$
Since by assumption $B$ is densely-defined, we are done.
\end{proof}

As a direct application (taken with $B=D_j$ and $C=x$), we get a useful result:
\begin{prop}\label{exp partial diff lemma} Let $1\leq j\leq d$ and let $x\in L_\infty(\qr).$ If $\partial_jx\in L_\infty(\qr),$ then 
$$\partial_j(e^{{\rm i}x})={\rm i}\int_0^1e^{{\rm i}(1-t)x}(\partial_jx)e^{{\rm i}tx}dt.$$
In particular, $\partial_j(e^{{\rm i}x})\in L_\infty(\qr).$
\end{prop}

Given $n\in\mathbb{N}.$ Consider the $n$-simplex 
$\Delta_n=\{\mathbf{s}=(s_0,,\cdots,s_n)\in\mathbb{R}_+^{n+1}:\sum_{j=0}^{n}s_j=1\},$
and set
$R_n=\{(s_0,\cdots,s_{n-1})\in\mathbb{R}_+^n:\sum_{j=0}^{n-1}s_j\leq1\}.$
Let $\sigma_n$ denote the standard measure on $\Delta_n$ determined by 
$$\int_{\Delta_n}f(s_0,\cdots,s_n)d\sigma_n(\mathbf{s})=\int_{R_n}f(s_0,\cdots,s_{n-1},1-\sum_{j=0}^{n-1}s_j)ds_0\cdots ds_{n-1},$$
for every continuous complex-valued function $f$ on $\mathbb{R}^{n+1}.$ 
\begin{lem}\label{fxbeta partial diff lemma} Assume $n,N\in\mathbb{N}.$ Let $\beta=(\beta_1,\cdots,\beta_n)\in(\mathbb{Z}_+^d)^n$ be such that $|\beta_k|_1\leq N$ for all $1\leq k\leq n.$ Let $x\in W^{N+1}_{\infty}(\qr)$ and set
$$ F_{x,\beta}({\bf s}):=F_{x;\beta_1,\cdots,\beta_n}({\bf s}):=e^{{\rm i}s_0x}\cdot\Big(\prod_{k=1}^{n} \partial^{\beta_k}x\cdot e^{{\rm i}s_kx}\Big), \quad {\bf s}=(s_0,\cdots,s_n)\in\Delta_n.$$
Then for every $1\leq j\leq d,$ we have
$$\begin{aligned}
\partial_j(F_{x,\beta}(
\mathbf{s}))&=\sum_{k=1}^nF_{x;\beta_1,\cdots,\beta_{k-1},\beta_k+e_j,\beta_{k+1},\cdots,\beta_n}({\bf s})\\&+{\rm i}\sum_{k=0}^n\int_0^{s_k}F_{x;\beta_1,\cdots,\beta_k,e_j,\beta_{k+1},\cdots,\beta_n}(s_0,\cdots,s_{k-1},s,s_k-s,s_{k+1},\cdots,s_n)ds.
\end{aligned}
$$
Here, $e_j\in\mathbb{Z}_+^d$ with $1$ in the $j$-th component and zeros elsewhere. 
\end{lem}
\begin{proof} Fix $j.$ A direct calculation gives
$$\partial_j(F_{x,\beta}({\bf s}))=F_{x;\beta_1,\cdots,\beta_{k-1},\beta_k+e_j,\beta_{k+1},\cdots,\beta_n}({\bf s})+\sum_{k=0}^nG_{k,x;\beta_1,\cdots,\beta_n}({\bf s}),$$
where, for every $0\leq k\leq n,$ $G_{k,x,\beta}(\mathbf{s})$ is defined by replacing the factor $e^{{\rm i}s_kx}$ in $F_{x,\beta}(\mathbf{s})$ with $\partial_j(e^{{\rm i}s_kx}).$
Since by Proposition \ref{exp partial diff lemma},
$$\partial_j(e^{{\rm i}s_kx})={\rm i}s_k\int_0^{1}e^{{\rm i}ss_kx}(\partial_jx) e^{{\rm i}(1-s)s_kx}ds={\rm i}\int_0^{s_k}e^{{\rm i}sx}( \partial_jx) e^{{\rm i}(s_k-s)x}ds,$$
it follows that
$$
G_{k,x,\beta}({\bf s})
={\rm i}\int_0^{s_k}F_{x;\beta_1,\cdots,\beta_k,e_j,\beta_{k+1},\cdots,\beta_n}(s_0,\cdots,s_{k-1},s,s_k-s,s_{k+1},\cdots,s_n)ds.
$$
The proof is complete.
\end{proof}

\begin{lem}\label{exp repeated diff lemma} Let $N\in\mathbb{N}$ and let $x\in W_\infty^N(\qr).$ Then for any $\alpha\in\mathbb{Z}_+^d$ with $|\alpha|_1\leq N,$ we have
$$\partial^{\alpha}(e^{{\rm i}x})=\sum_{m=1}^{|\alpha|_1}\sum_{\substack{\beta\in(\mathbb{Z}^d_+)^m\\ |\beta|_1=|\alpha|_1}}c_{\beta}\int_{\Delta_m}F_{x,\beta}({\bf s})d\sigma_m({\bf s}).$$
\end{lem}

\begin{proof} We shall prove this result by induction on $N.$ Base of induction (i.e., the case $N=1$) is established in Lemma \ref{exp partial diff lemma}. 
Let us now establish the step of induction.	Fix $1\leq j\leq d$ such that $\alpha\geq e_j$ and suppose the result is valid for $\alpha-e_j.$ It is clear that $\partial^\alpha(e^{{\rm i}x})=\partial_j(\partial^{\alpha-e_j}(e^{{\rm i}x})).$
By the inductive assumption,
$$\partial^\alpha(e^{{\rm i}x})=\sum_{m=1}^{|\alpha|_1-1}\sum_{\substack{\beta\in(\mathbb{Z}^d_+)^m\\ |\beta|_1=|\alpha|_1-1}}c_{\beta} \int_{\Delta_m}\partial_j(F_{x,\beta}({\bf s}))d\sigma_m({\bf s}).$$
Using obvious equality for $\gamma\in(\mathbb{Z}_+^d)^{m+1},$
\small{$$
	\int_{\Delta_m}\big(\int_0^{s_k}F_{x,\gamma}(s_0,\cdots,s_{k-1},s,s_k-s,s_{k+1},\cdots,s_m)ds\big)d\sigma_m({\bf s})=\int_{\Delta_{m+1}}F_{x,\gamma}({\bf s})d\sigma_{m+1}({\bf s})$$}
and Lemma \ref{fxbeta partial diff lemma} yields the desired result. 

\end{proof}
\begin{defi}
For every $m\in\mathbb{Z}_+,$ let $\mathcal{W}^m(\mathbb{R})$ denote the Wiener space consisting of $m$ times continuously differentiable functions $f$ on $\mathbb{R}$ such that
$$\|f\|_{\mathcal{W}^m(\mathbb{R})}:=\int_{\mathbb{R}}|\widehat{f^{(m)}}(t)|dt=\int_{\mathbb{R}}|t|^m|\hat{f}(t)|dt<\infty.$$
In particular, if $f\in\mathcal{W}^m(\mathbb{R})$ is compactly supported, then 
$f\in\mathcal{W}^k(\mathbb{R}),$ for all $ 0\leq k\leq m.$ This follows from elementary computations; the details are omitted.
\end{defi}

In what follows, fix $n\in\mathbb{N}$ and let $\{E_{k,l}\}^n_{k,l=1}\subset M_n(\mathbb{C})$ be matrix units. Then for $X\in M_n(\mathbb{C})\otimes L_\infty(\qr),$ we may write $X=\sum_{k,l=1}^nE_{k,l}\otimes X_{k,l},$ with $X_{k,l}\in L_\infty(\qr).$  
The following inequalities are immediate
$$
\max_{1\leq k,l\leq n}\|X_{k,l}\|_{L_p(\qr)}\leq\|X\|_{L_p(M_n(\mathbb{C})\bar{\otimes }L_\infty(\qr))}\leq\sum_{k,l=1}^n\|X_{k,l}\|_{L_p(\qr)},\quad 1\leq p\leq\infty.
$$
Furthermore, for any $\alpha\in\mathbb{Z}_+^d$ we have
\begin{equation}\label{L_p norm estimate}
({\rm id}\otimes\partial^\alpha)(X)\in L_p(M_n(\mathbb{C})\bar{\otimes}L_\infty(\qr))\Longleftrightarrow \partial^\alpha X_{k,l}\in L_p(\qr),\quad \forall 1\leq k,l\leq n.
\end{equation}

\begin{lem}\label{standard wiener lemma} Let $X=X^{\ast}\in M_n(\mathbb{C})\otimes W_1^d(\mathbb{R}_{\theta}^d).$ If $f\in \mathcal{W}^d(\mathbb{R})$ is compactly supported, then $f(X)\in M_n(\mathbb{C})\otimes W_1^d(\mathbb{R}_{\theta}^d).$
\end{lem}
\begin{proof}Let $\alpha\in\mathbb{Z}_+^d$ be such that $|\alpha|_1\leq d.$ 
We start with the case $n=1.$ Now, $X=X^\ast\in W_1^d(\qr).$ 	
By Fourier inversion and functional calculus,
$$\partial^{\alpha}(f(X))=\partial^{\alpha}\Big(\int_{\mathbb{R}}\hat{f}(t)e^{{\rm i}tX}dt\Big)=\int_{\mathbb{R}}\hat{f}(t)\partial^{\alpha}(e^{{\rm i}tX})dt.$$
By Lemma \ref{exp repeated diff lemma}, 
$$\partial^{\alpha}(e^{{\rm i}tX})=\sum_{m=1}^{|\alpha|_1}\sum_{\substack{\beta\in(\mathbb{Z}^d_+)^m\\ |\beta|_1=|\alpha|_1}}c_{\beta}\int_{\Delta_m}F_{tX,\beta}({\bf s})d\sigma_m({\bf s}),\quad t\in\mathbb{R}.$$
Obviously, for every $\beta\in(\mathbb{Z}_+^d)^m,$
$$\|F_{tX,\beta}({\bf s})\|_1\leq |t|^m\|\partial^{\beta_1}X\|_1\|\partial^{\beta_2}X\|_{\infty}\cdots\|\partial^{\beta_m}X\|_{\infty}\lesssim_{\theta} |t|^m\|X\|_{W_1^d(\mathbb{R}_{\theta}^d)}^m.$$
Thus,
$\|\partial^{\alpha}(e^{itX})\|_1\lesssim_{\theta,\alpha} \sum_{m=1}^{|\alpha|_1} |t|^m\|X\|_{W_1^d(\mathbb{R}_{\theta}^d)}^m.$
Consequently,
$$
\begin{aligned}
\|\partial^{\alpha}(f(X))\|_1&\leq\int_{\mathbb{R}}|\hat{f}(t)|\cdot\|\partial^{\alpha}(e^{{\rm i}tX})\|_1dt\\
&\lesssim\max_{1\leq m\leq|\alpha|_1}\int_{\mathbb{R}}|t|^m|\hat{f}(t)|dt\cdot \sum_{m=1}^{|\alpha|_1}\|X\|_{W_1^d(\mathbb{R}_{\theta}^d)}^m
\end{aligned}
$$
By assumptions on $f$ and $X,$ the right hand side is finite.
This proves the result for $n=1.$
Suppose now $n>1.$ 
Clearly, $M_n(\mathbb{C})\otimes L_\infty(\qr)$ is a $C^\ast$-algebra.
Thus,
$f(X)\in M_n(\mathbb{C})\otimes L_\infty(\qr).$
Repeating the previous argument, we obtain
$$\|({\rm id}\otimes\partial^{\alpha})(f(X))\|_1\lesssim\sum_{m=1}^{|\alpha|_1}\sum_{\substack{\beta\in(\mathbb{Z}^d_+)^m\\ |\beta|_1=|\alpha|_1}} \|\partial^{\beta_1}X\|_1\|\partial^{\beta_2}X\|_{\infty}\cdots\|\partial^{\beta_m}X\|_{\infty}.$$
General result now follows from \eqref{L_p norm estimate}. 
\end{proof}

\begin{thm}\label{residue-matrix case}
Let $n\in \mathbb{N}.$	Let $A=\mathrm{id}\otimes\pi_{2}(g) (1-\Delta_{\theta})^{-\frac12}$ with $0\leq g \in C^{\infty}(\mathbb{S}^{d-1}),$ and let $0\leq B\in M_n(\mathbb{C})\otimes W_1^d(\mathbb{R}_{\theta}^{d}).$ Then the function 
$$z\mapsto (\mathrm{tr}\otimes\mathrm{Tr})\left(A^{z} B^{z}\right),\quad\Re(z)>d$$ 
allows meromorphic extension to the half-plane $\{\Re(z)>1\}$ with poles at $2,4,\cdots,d.$ The residue at $z=d$ is explicitly computable as follows:
$$\underset{z=d}{\rm Res}~(\mathrm{tr}\otimes\mathrm{Tr})(A^zB^z)=d\big(\kappa_d\|B\|_d\|g\|_d\big)^d.$$
\end{thm}
\begin{proof} Fix $\phi\in C^{\infty}_c(\mathbb{R})$ such that $\phi=1$ on $[0,\|x\|_{\infty}].$ For every $z\in\mathbb{C}$ with $\Re(z)>d,$ define the function $f_z$ on $\mathbb{R}$ by setting $f_z(t)=|t|^z\phi(t),$ $ t\in\mathbb{R}.$ By the Leibniz rule and Lemma \ref{second zigmund lemma}, $\widehat{f^{(d)}_z}\in L_1(\mathbb{R}).$ In other words, $f_z\in\mathcal{W}^d(\mathbb{R}).$ Since also $f_z$ is compactly supported, it follows from Lemma \ref{standard wiener lemma} that $B^z=f_z(B)\in M_n(\mathbb{C})\otimes W_1^d(\mathbb{R}_\theta^d).$
Define
$$h_z(t)=g(\frac{t}{|t|})^z(1+|t|^2)^{-\frac{z}{2}},\quad  t\in\mathbb{R}^d.$$
Clearly, if $\Re(z)>d,$ then $h_z\in\ell_1 (L_\infty)(\mathbb{R}^d).$ By tracial property and by Corollary \ref{mixed trace thm-matrix case},
\begin{equation}\label{splitting  into products}
\begin{split}
	({\rm tr}\otimes\mathrm{Tr})(A^z B^z)&=({\rm tr}\otimes\mathrm{Tr})(B^zA^z)\\&=({\rm tr}\otimes\mathrm{Tr})(B^z \big({\rm id}\otimes h_z(\nabla_{\theta}))\big)\\&=\frac{1}{(2\pi)^d}({\rm tr}\otimes\tau_{\theta})\left(B^z\right)\int_{\mathbb{R}^d} h_z (t)dt.
\end{split}
\end{equation}
Passing to the polar coordinates, we obtain
$$\int_{\mathbb{R}^d} h_z(t)dt=\int_{\mathbb{S}^{d-1}}g^{z}(s)ds\cdot\int_{0}^{\infty}(1+r^2)^{-\frac{z}{2}}r^{d-1}dr.$$
A standard computation yields that
$$
\begin{aligned}
\int_0^{\infty}(1+r^2)^{-\frac{z}{2}}r^{d-1}dr&\stackrel{r=u^{\frac12}}{=}\frac12\int_0^{\infty}(1+u)^{-\frac{z}{2}}u^{\frac{d}{2}-1}du\\
&\stackrel{u=\frac{v}{1-v}}{=}\frac12\int_0^1v^{\frac{d}{2}-1}(1-v)^{\frac{z-d}{2}-1}dv=\frac12B(\frac{d}{2},\frac{z-d}{2}),
\end{aligned}$$
where $B(\cdot,\cdot)$ is the Beta function.
Further, we have
$$B(\frac{d}{2},\frac{z-d}{2})=\frac{2^{\frac{d}{2}}\Gamma(\frac{d}{2})}{(z-2)(z-4)\cdots(z-d)}.$$
Consequently,
\begin{equation}\label{h_z computation}
\int_{\mathbb{R}^d} h_z(t)dt=\frac{2^{\frac{d}{2}-1}\Gamma(\frac{d}{2})}{(z-2)(z-4)\cdots(z-d)}\cdot\int_{\mathbb{S}^{d-1}}g^{z}(s)ds.
\end{equation}
Combining equations \eqref{splitting  into products} and \eqref{h_z computation}, we claim that
$$ ({\rm tr}\otimes\mathrm{Tr})(A^z B^z)=\frac{d\kappa_d^d}{(z-2)(z-4)\cdots(z-d)}\cdot({\rm tr}\otimes\tau_{\theta})\left(B^z\right)\int_{\mathbb{S}^{d-1}}g^{z}(s)ds.$$
Since $B\in M_n(\mathbb{C})\otimes W_1^d(\qr)\subset M_n(\mathbb{C})\otimes L_1(\qr),$ we have $({\rm id}\otimes r_\theta)(B)\in M_n(\mathbb{C})\otimes\mathcal{L}_1(L_2(\mathbb{R}^{\frac{d}{2}})),$ it turns out the function $z\mapsto({\rm tr}\otimes\tau_{\theta})(B^z)$
is holomorphic in the half-plane $ \{\Re(z)>1\}.$ In addition, the function 
$z\mapsto\int_{\mathbb{S}^{d-1}} g^z(s)ds$
is holomorphic in the half-plane $\{\Re(z)>0\}.$ The proof is complete.
\end{proof}

\section{Spectral asymptotics\textemdash a special case}\label{sec-asymptotic special case}Our main task in this section is to prove the following result.
\begin{thm}\label{key thm}
Let $0\leq g\in C(\mathbb{S}^{d-1})$ and let $T\in M_{n}(\mathbb{C})\otimes \mathcal{S}(\qr).$ We have
$$\lim_{t\to\infty}t^{\frac1d}\mu\big(t,T\cdot (\mathrm{id}\otimes \pi_2(g)(1-\Delta_{\theta})^{-\frac12})\big)= \kappa_d\left\|T\right\|_{L_d(M_{n}(\mathbb{C})\bar{\otimes} L_{\infty}(\qr))}\|g\|_{L_d(\mathbb{S}^{d-1})}.$$
Here, the constant $\kappa_d$ is given in Theorem \ref{main1}.
\end{thm}

The key device at hands is a Wiener-Ikehara Tauberian theorem, see Section III.4 of \cite{K2004}.

\begin{thm}\label{wiener-ikehara thm} Let $p>0$ and let $0\leq V\in\mathcal{L}_{p,\infty}.$ If, for some $c_V>0,$ the function 
$$z\to {\rm Tr}(V^z)-\frac{c_V}{z-p},\quad \Re(z)>p,$$
extends continuously to the closed half-plane $\{\Re(z)\geq p\},$ then
$$\lim_{t\to\infty}t^{\frac1p}\mu(t,V)=(\frac{c_V}{p})^{\frac1p}.$$
\end{thm}

The condition below first appeared in \cite{SZ2023} (see Condition 5.4.1 there).

\begin{cond}\label{asterisque condition} Let $p>2$ and let $0\leq A,B\in \mathcal{L}_{\infty}$ satisfy the following conditions:
\begin{enumerate}[{\rm (i)}]
\item\label{aca} $A^pB\in\mathcal{L}_{1,\infty},$
\item\label{acb} $A^{q-2}[A,B]\in\mathcal{L}_1$ for every $q>p,$
\item\label{acc} $B^{\frac12}AB^{\frac12}\in \mathcal{L}_{p,\infty},$ 
\item\label{acd} $[A,B^{\frac12}]\in \mathcal{L}_{\frac p 2,\infty}.$
\end{enumerate}
\end{cond}

\begin{thm}\label{nc tauberian theorem} Let $p>2$ and let $0\leq A,B\in \mathcal{L}_{\infty}$ satisfy Condition \ref{asterisque condition}. If there is a constant $c>0$ such that the function
\begin{equation*}
z\mapsto {\rm Tr}( A^z B^z)-\frac{c}{z-p} , \quad z\in\mathbb{C},\quad \Re(z)>p,
\end{equation*}
admits a continuous extension to the closed half plane $\{z\in \mathbb{C}:\Re(z)\geq p \},$ then
\begin{equation*}
\lim_{t\rightarrow\infty}t^{\frac1p}\mu(t,AB)=\big(\frac{c}{p}\big)^{\frac1p}.
\end{equation*}
\end{thm}
\begin{proof} Under Condition \ref{asterisque condition} above, Theorem 5.4.2 in \cite{SZ2023} asserts that, the function
$$z\to {\rm Tr}(A^zB^z)-{\rm Tr}((B^{\frac12}AB^{\frac12})^z),\quad \Re(z)>p,$$
admits an analytic continuation to the half-plane $\{\Re(z)>p-1\}.$ This fact together with the current assumptions implies, the function
$$z\mapsto\mathrm{Tr}((B^{\frac12}AB^{\frac12})^z)-\frac{c}{z-p},\quad \Re(z)>p,$$ 
has a continuous continuation to the half-plane $\{\Re(z)\geq p\}.$ Appealing to Theorem \ref{wiener-ikehara thm}, we have
$$\lim_{t\to\infty}t^{\frac1p}\mu(t,B^{\frac12}AB^{\frac12})=(\frac{c}{p})^{\frac{1}{p}}.$$
Note that, by Condition \ref{asterisque condition} \eqref{acd}, $AB-B^{\frac12}AB^{\frac12}=[A,B^{\frac12}]B^{\frac12}\in(\mathcal{L}_{p,\infty})_0.$ The result now follows from Lemma \ref{bs sep lemma}. 
\end{proof}

\begin{cor}\label{nc tauberian corollary} Let $0\leq X,Y\in \mathcal{L}_{\infty}$ be such that $[X,Y^{\frac12}]\in(\mathcal{L}_{2,\infty})_0.$ Let $A=X^{\frac12}$ and $B=Y^{\frac12}$ satisfy Condition \ref{asterisque condition} for $p=4.$ If, for some $c>0,$ the function
$$z\mapsto{\rm Tr}(X^zY^z)-\frac{c}{z-2},\quad \Re(z)>2,$$
admits a continuous extension to the closed half-plane $\{z\in \mathbb{C}:\Re(z)\geq 2\}$, then 
\begin{equation}\label{Taub cor}
\lim _{t\rightarrow \infty}t^{\frac12}\mu (t,XY)=\big(\frac{c}{2}\big)^{\frac12}.
\end{equation}
\end{cor}

\begin{proof} By assumption, the function 
$$z\mapsto{\rm Tr}(A^zB^z)-\frac{2c}{z-4},\quad \Re(z)>4,$$
admits a continuous continuation to the closed half-plane $\{z\in \mathbb{C}:\Re(z)\geq 4\}.$

Hence, the operators $A$ and $B$ satisfy the assumption in Theorem \ref{nc tauberian theorem} with $p=4$, it follows that
$\lim_{t\rightarrow \infty}t^{\frac14}\mu (t,X^{\frac12}Y^{\frac12})=\lim_{t\rightarrow \infty}t^{\frac14}\mu (t,AB)=\big(\frac{2c}{4}\big)^{\frac14}=\big(\frac{c}{2}\big)^{\frac14}.$
Observing that $ \mu(Y^{\frac12}XY^{\frac12})=\mu((X^{\frac12}Y^{\frac12})^{\ast}(X^{\frac12}Y^{\frac12}))=\mu(X^{\frac12}Y^{\frac12})^2,$ we obtain
$$ \lim_{t\rightarrow \infty}t^{\frac12}\mu (t,Y^{\frac12}XY^{\frac12})=\big(\frac{c}{2}\big)^{\frac12}.$$
Since by assumption, $XY-Y^{\frac12}XY^{\frac12}=[X,Y^{\frac12}]\cdot Y^{\frac12}\in(\mathcal{L}_{2,\infty})_0,$
the result follows immediately from Lemma \ref{bs sep lemma}.
\end{proof}

The next result \cite[Lemma 2.3]{LSZ2020} allows us to ignore matrix tensor factors in weak Schatten class estimates.
\begin{lem}\label{matrix factor ignored}
If $x\in M_n(\mathbb{C})$ and $y\in\mathcal{L}_{p,\infty}(H),$ then $x\otimes y\in\mathcal{L}_{p,\infty}(\mathbb{C}^n\otimes H)$ and
$$\|x\otimes y\|_{p,\infty}\leq \|x\|_p\|y\|_{p,\infty},\quad p>0.$$ 
\end{lem}
In the next two lemmas, we verify the conditions in Theorem \ref{nc tauberian theorem} and Corollary \ref{nc tauberian corollary} for sufficiently good functions and operators. We start with the case $d>2.$ 
\begin{lem}\label{verification of the asterisque conditions matrix case} Assume $d>2.$ Let $A=\mathrm{id}\otimes\pi_2(g)(1-\Delta_{\theta})^{-\frac12}$ with $0\leq g\in C^{\infty}(\mathbb{S}^{d-1}),$ and let $0\leq B\in M_n(\mathbb{C})\otimes\mathcal{S}(\qr).$  Then the operators $A$ and $B$ satisfy Condition \ref{asterisque condition} for $p=d.$
\end{lem}
\begin{proof} Let $\{E_{i,j}\}_{i,j=1}^n\subset M_n(\mathbb{C})$ be matrix units. We write 
$$A^d=\mathrm{id}\otimes\pi_2(g^d)(1-\Delta_{\theta})^{-\frac{d}{2}},\quad B=\sum_{i,j} E_{i,j}\otimes B_{i,j},\quad B_{i,j}\in\mathcal{S}(\qr).$$
Since $\pi_2(g)$ commutes with $(1-\Delta_{\theta})^{-\frac12},$ we have	
$$A^dB=\sum_{i,j }E_{i,j}\otimes\pi_2(g^d) (1-\Delta_{\theta})^{-\frac{d}{2}}B_{i,j}.$$
By taking $\beta=d$ in Lemma \ref{arbitrary index for Bessel}, we deduce that for  all $1\leq i,j\leq n,$
$$\pi_2(g^d)\cdot  (1-\Delta_{\theta})^{-\frac{d}{2}}B_{i,j}\in\mathcal{L}_{1,\infty}.$$
In view of Lemma \ref{matrix factor ignored},
we have $A^dB\in\mathcal{L}_{1,\infty}.$ This verifies the Condition \ref{asterisque condition} \eqref{aca}.

For every $q>d,$ we have
$$A^{q-2}[A,B]=
\sum_{i,j} E_{i,j}\otimes\pi_2(g^{q-2}) (1-\Delta_{\theta})^{1-\frac{q}{2}}[\pi_2(g)(1-\Delta_{\theta})^{-\frac12},B_{i,j}].$$
Taking $\beta=q-2$ and $\alpha=-1$ in Theorem \ref{commutator-PSD}, we conclude that
$$\pi_2(g^{q-2})\cdot (1-\Delta_{\theta})^{1-\frac{q}{2}}[\pi_2(g)(1-\Delta_{\theta})^{-\frac12},B_{i,j}]\in\mathcal{L}_{\infty}\cdot\mathcal{L}_{\frac{d}{q},\infty}\subset\mathcal{L}_1.$$ 
Thus, $A^{q-2}[A,B]\subset\mathcal{L}_1.$ This verifies the Condition \ref{asterisque condition} \eqref{acb}.

By Lemma \ref{root lemma matrix case}, we have $B^{\frac12}\in M_n(\mathbb{C})\otimes\mathcal{S}(\qr).$ Write $B^{\frac12}=\sum_{i,j }E_{i,j}\otimes(B^{\frac12})_{i,j},$ with $(B^{\frac12})_{i,j}\in\mathcal{S}(\qr).$ By Lemma \ref{arbitrary index for Bessel},
$(1-\Delta_{\theta})^{-\frac12} (B^{\frac12})_{i,j}\in\mathcal{L}_{d,\infty}$ for $ 1\leq i,j\leq n.$
Hence,
$$B^{\frac12}AB^{\frac12}=\sum_{i,j,k,l} E_{i,j}E_{k,l}\otimes(B^{\frac12})_{i,j}\pi_2(g) (1-\Delta_{\theta})^{-\frac12} (B^{\frac12})_{k,l}\in\mathcal{L}_{d,\infty}.$$
This verifies the Condition \ref{asterisque condition} \eqref{acc}. Further, we write
$$[A,B^{\frac12}]=\sum_{i,j} E_{i,j}\otimes[\pi_2(g)(1-\Delta_{\theta})^{-\frac12},(B^{\frac12})_{i,j}].$$
Taking $\alpha=-1$ and $\beta=0$ in Theorem \ref{commutator-PSD}, we obtain
$$[\pi_2(g)(1-\Delta_{\theta})^{-\frac12},(B^{\frac12})_{i,j}]\in\mathcal{L}_{\frac{d}{2},\infty},\quad 1\leq i,j\leq n.$$
It follows that $[A,B^{\frac12}]\in\mathcal{L}_{\frac{d}{2},\infty}.$
This verifies the Condition \ref{asterisque condition} \eqref{acd}.
\end{proof}

The case $d=2$ is handled as follows:
\begin{lem}\label{verification of the asterisque conditions matrix case d=2} Let $g\in C^{\infty}(\mathbb{S}^1)$ be strictly positive. Let $X=\mathrm{id}\otimes\pi_2(g)(1-\Delta_{\theta})^{-\frac12}$ and $0\leq Y\in M_n(\mathbb{C})\otimes\mathcal{S}(\mathbb{R}^2_\theta).$ Then $[X,Y^{\frac12}]\in(\mathcal{L}_{2,\infty})_0,$ and the operators $A:=X^{\frac12}$ and $B:=Y^{\frac12}$ satisfy Condition \ref{asterisque condition} for $p=4.$ 
\end{lem}
\begin{proof} Verification of Condition \ref{asterisque condition} \eqref{aca}, \eqref{acc} is identical to the one for $d>2$ (see the proof of Lemma \ref{verification of the asterisque conditions matrix case} above) and is, therefore, omitted.

By Lemma \ref{root lemma matrix case}, $B\in M_n(\mathbb{C})\otimes \mathcal{S}(\mathbb{R}^2_\theta).$ We may write $B=\sum_{i,j} E_{i,j}\otimes B_{i,j},$ with $B_{i,j}\in\mathcal{S}(\mathbb{R}^2_\theta)$ for all $i,j=1,\cdots,n.$ Then for every $q>4,$ we have
\begin{equation*}
A^{q-2}[A,B]=\sum_{i,j} E_{i,j}\otimes\pi_2(g^{\frac{q}{2}-1}) (1-\Delta_{\theta})^{\frac12-\frac{q}{4}}[\pi_2(g^{\frac12})(1-\Delta_{\theta})^{-\frac14},B_{i,j}].
\end{equation*}
Since $g\in C^{\infty}(\mathbb{S}^1)$ is strictly positive, it follows that $g^{\frac12}\in C^{\infty}(\mathbb{S}^1).$  
Taking $\alpha=-\frac12,$ $\beta=\frac{q}{2}-1$ in Theorem \ref{commutator-PSD}, we infer that for all $1\leq i,j\leq n,$
$$\pi_2(g^{\frac{q}{2}-1})\cdot (1-\Delta_{\theta})^{\frac12-\frac{q}{4}}[\pi_2(g^{\frac12})(1-\Delta_{\theta})^{-\frac14},B_{i,j}]\in\mathcal{L}_{\infty}\cdot\mathcal{L}_{\frac{4}{q+1}}\subset\mathcal{L}_1.$$
Thus, $A^{q-2}[A,B]\in\mathcal{L}_1.$
This verifies the Condition \ref{asterisque condition} \eqref{acb}.

Again by Lemma \ref{root lemma matrix case}, $B^{\frac12}\in M_n(\mathbb{C})\otimes\mathcal{S}(\mathbb{R}^2_\theta).$
Taking $\alpha=-\frac12$ and $\beta=0$ in Theorem \ref{commutator-PSD}, we obtain
\begin{equation*}
[A,B^{\frac12}]=\sum_{i,j} E_{i,j}\otimes[\pi_2(g^{\frac12})(1-\Delta_{\theta})^{-\frac14},(B^{\frac12})_{i,j}]\in\mathcal{L}_{\frac43,\infty}\subset\mathcal{L}_{2,\infty}.
\end{equation*}
This verifies the Condition \ref{asterisque condition} \eqref{acd}.
The result $[X,Y^{\frac12}]\in(\mathcal{L}_{2,\infty})_0$ follows directly from Theorem \ref{commutator-PSD} by taking $\alpha=-1$ and $\beta=0$ there.
\end{proof}

Having done these verifications, 
we are now able to apply the Tauberian theorem to establish the following spectral asymptotics.
\begin{lem}\label{section 6 main lemma} Let $d\geq2.$ If $0\leq T\in M_n(\mathbb{C})\otimes\mathcal{S}(\qr)$ and $0\leq g\in C^{\infty}(\mathbb{S}^{d-1}),$ then 
$$\lim_{t\to\infty}t^{\frac1d}\mu\big(t,T\big(\mathrm{id}\otimes\pi_2(g)(1-\Delta_{\theta})^{-\frac12}\big)\big)=\kappa_d \left\|T\right\|_{L_d(M_n(\mathbb{C})\bar{\otimes} L_\infty(\qr))}\|g\|_{L_d(\mathbb{S}^{d-1})}.$$
\end{lem}
\begin{proof}
In the case $d>2,$ we take $A=1\otimes\pi_2(g)(1-\Delta_{\theta})^{-\frac12}$ and $B=T.$ By Lemma \ref{verification of the asterisque conditions matrix case} and Corollary \ref{residue-matrix case}, the operators $A$ and $B$ satisfy the assumption of Theorem \ref{nc tauberian theorem} with $p=d$ and $c=d\kappa_d^d.$ It follows that
$$\lim_{t\to\infty}t^{\frac1d}\mu(AB)=\kappa_d \left\|T\right\|_{L_d(M_n(\mathbb{C})\bar{\otimes} L_\infty(\qr))}\|g\|_{L_d(\mathbb{S}^{d-1})}.$$
Note that by Theorem \ref{commutator-PSD} we have $AB-BA\in\mathcal{L}_{\frac{d}{2},\infty}\subset(\mathcal{L}_{d,\infty})_0.$
Applying Lemma \ref{bs sep lemma} proves the statement for $d>2.$
It remains to deal with the case $d=2.$ 
For this purpose, we set
$X:=\mathrm{id}\otimes\pi_2(g)(1-\Delta_{\theta})^{-\frac12},$ $Y:=T,$
$g_m:=g+\frac{1}{m},~m\in\mathbb{N},$ and $ X_m:=\mathrm{id}\otimes\pi_2(g_m)(1-\Delta_{\theta})^{-\frac12}.$
By Lemma \ref{verification of the asterisque conditions matrix case d=2} 
and Corollary \ref{residue-matrix case}, the operators $X_m$ and $Y$ with $c=2(\kappa_2\|T\|_2\|g_m\|_2)^2,$ satisfy the assumption in Corollary \ref{nc tauberian corollary}. It then follows that
$$
\lim_{t\to\infty}t^{\frac12}\mu(t,YX_m)=(\frac{c}{2})^{\frac12}=\kappa_2\|T\|_2\|g_m\|_2 .
$$
Write $Y=\sum_{i,j}E_{i,j}\otimes T_{i,j}$ with $T_{i,j}\in\mathcal{S}(\qr)$ for all $1\leq i,j\leq n.$ By the quasi-triangle inequality, we have
$
\|YX_m-YX\|_{2,\infty}\lesssim \frac{1}{m} \max_{i,j}\|T_{i,j}(1-\Delta_{\theta})^{-\frac12}\|_{2,\infty}\to0,$ as $m\to\infty.$
Meanwhile, note that by construction $\|g_m\|_2\to\|g\|_2$ as $m\to\infty.$
Combining these facts with Lemma \ref{another limit lemma}, we conclude that
$$\lim_{t\to\infty}t^{\frac12}\mu(t,YX)=\kappa_2\|T\|_2\|g\|_2.$$
Note that by Theorem \ref{commutator-PSD} we have $YX-XY\in\mathcal{L}_{1,\infty}\subset(\mathcal{L}_{2,\infty})_0.$ Applying Lemma \ref{bs sep lemma} proves the statement for $d=2.$
\end{proof}

Finally, let us remove the smoothness restriction on functions.
\begin{proof}[Proof of Theorem \ref{key thm}] Recall that $0\leq g\in C(\mathbb{S}^{d-1})$ and $T\in M_{n}(\mathbb{C})\otimes \mathcal{S}(\qr).$ Choose $0\leq g_m \in C^{\infty}(\mathbb{S}^{d-1}),$ $m\in\mathbb{N}$ such that $g_m \rightarrow g $ in $C(\mathbb{S}^{d-1}).$ 
Set
$$ A_m=\mathrm{id}\otimes\pi_2(g_m)(1-\Delta_{\theta})^{-\frac12},\quad A=\mathrm{id}\otimes\pi_2(g)(1-\Delta_{\theta})^{-\frac12}.$$
Note that $|T|^2=T^{\ast}T\in M_n\otimes\mathcal{S}(\qr).$ Then by Lemma \ref{root lemma matrix case}, $|T|\in M_n\otimes\mathcal{S}(\qr).$ 
By Lemma \ref{section 6 main lemma}, for every $m\geq1$ we have
\begin{equation}\label{pre-asymptotic}
\lim_{t\to0}t^{\frac1d}\mu(t,TA_{m})=\lim_{t\to0}t^{\frac1d}\mu(t,|T|A_m)=\kappa_d\|T\|_d\|g_m\|_d.
\end{equation}
Writing $T=(T_{i,j})_{1\leq i,j\leq n},$
with $T_{i,j}\in\mathcal{S}(\qr),$
we have 
$$\begin{aligned}
\|TA_m-TA\|_{d,\infty}&=\big\|\big(T_{i,j}(1-\Delta_{\theta})^{-\frac12}\cdot\pi_2(g_m-g)\big)_{1\leq i,j\leq n}\big\|_{d,\infty}\\
&\lesssim \max_{i,j}\left\|T_{i,j}(1-\Delta_{\theta})^{-\frac12}\right\|_{d,\infty}\left\|g_m-g\right\|_{\infty}.
\end{aligned}$$
It follows from Lemma \ref{arbitrary index for Bessel} that $TA_m\to TA$ in $\mathcal{L}_{d,\infty}.$
Clearly, $\|g_m\|_d\to \|g\|_d$ as $m\to\infty.$
Combining these facts with equation \eqref{pre-asymptotic} and Lemma \ref{another limit lemma} gives
$$\lim_{t\to\infty}t^{\frac{1}{d}}\mu(t,TA)=\kappa_d\|T\|_{d}\|g\|_{d}.$$
\end{proof}

\section{Spectral asympotics\textemdash general case}\label{sec-proof of main results}
In this section, we generalize Theorem \ref{key thm} to the case involving multiple continuous functions, and then complete the proof of our main results \textemdash Theorem \ref{main1} and Theorem \ref{main2}.

\begin{thm}\label{key thm plane}
Let $\{\Gamma_j\}_{j=1}^J\subset M_n(\mathbb{C}),$ $\{x_j\}_{j=1}^J\subset \mathcal{S}(\qr)$ and $\{g_j\}_{j=1}^J\subset C(\mathbb{S}^{d-1}).$
We have
$$\lim_{t\to\infty}t^{\frac1d}\mu\Big(t, \sum_{j=1}^J\Gamma_j\otimes\pi_1(x_j)\pi_2(g_j) (1-\Delta_{\theta})^{-\frac12}\Big)=\kappa_d\Big\| \sum_{j=1}^J\Gamma_j\otimes x_j\otimes g_j \Big\|_d,$$
where the $d$-norm is taken over $(M_n(\mathbb{C})\bar{\otimes}L_\infty(\qr)\bar{\otimes}L_\infty(\mathbb{S}^{d-1}
),\mathrm{tr}\otimes \tau_{\theta}\otimes\int).$
\end{thm}

Similarly to \cite{MSX2023}, we first deal with a special case where those continuous functions are pairwise disjointly supported. 
\begin{lem}\label{key thm disjoint plane} Let $\{\Gamma_j\}_{j=1}^J\subset M_n(\mathbb{C})$ and let $\{x_j\}_{j=1}^J\subset \mathcal{S}(\qr).$ Let $\{g_j\}_{j=1}^J\subset C(\mathbb{S}^{d-1})$ be non-negative functions. Assume that, for every $1\leq j_1,j_2\leq J$, either $g_{j_1}g_{j_2}=0$ or $g_{j_1}=g_{j_2}$. Under those assumptions, we have
$$\lim_{t\to\infty}t^{\frac1d}\mu\big(t, \sum_{j=1}^J\Gamma_j\otimes\pi_1(x_j)\pi_2(g_j) (1-\Delta_{\theta})^{-\frac12}\big)=\kappa_d\big\| \sum_{j=1}^J\Gamma_j\otimes x_j\otimes g_j\big\|_d.$$
\end{lem}

\begin{proof} Define equivalence relation on $\{1,\cdots,J\}$ by setting $j_1\sim j_2$ if $g_{j_1}=g_{j_2}.$ Split $\{1,\cdots,J\}$ into equivalence classes $\{C_k\}_{k=1}^K.$ Fix $m_k\in C_k,$ $1\leq k\leq K,$ and set
$$A_k=\sum_{j\in C_k}\Gamma_j\otimes\pi_1(x_j),\quad B_k=\mathrm{id}\otimes\pi_2(g_{m_k})(1-\Delta_{\theta})^{-\frac12}.$$
For each $1\leq k\leq K,$ we have
\small{$$A_kB_k=\sum_{j\in C_k}\Gamma_j \otimes\pi_1(x_j)\pi_2(g_{m_k})(1-\Delta_{\theta})^{-\frac12},\quad B_kA_k=\sum_{j\in C_k}\Gamma_j \otimes\pi_2(g_{m_k})(1-\Delta_{\theta})^{-\frac12}\pi_1(x_j).$$}
In particular, we may write
$\sum_{j=1}^J\Gamma_j\otimes\pi_1(x_j)\pi_2(g_j) (1-\Delta_{\theta})^{-\frac12}=\sum_{k=1}^KA_kB_k.$
By Theorem \ref{key thm}, for every $1\leq k\leq K,$
$$\lim_{t\to\infty}t^{\frac1d}\mu(t,A_kB_k)=\kappa_d\Big\|\sum_{j\in C_k}\Gamma_j\otimes x_j\Big\|_{d}\|g_{m_k}\|_{d}.$$
Hence,
$$\lim_{t\to\infty}t^{\frac2d}\mu(t,|A_kB_k|^2)=\kappa^2_d\Big\|\sum_{j\in C_k}\Gamma_j\otimes x_j\Big\|_{d}^2\|g_{m_k}\|_{d}^2.$$
By construction, the  operators $\{|A_kB_k|^2\}_{k=1}^K$ are pairwise orthogonal. Indeed, for $1\leq k\neq l\leq K$ we have
$
|A_kB_k|^2|A_lB_l|^2=B^{\ast}_kA_k^{\ast}A_kB_kB^{\ast}_lA_l^{\ast}A_lB_l=0$ since $ B_kB^{\ast}_l=\mathrm{id}\otimes\pi_2(g_{m_k}g_{m,l})(1-\Delta_{\theta})^{-1}=0.
$
Now applying Lemma \ref{bs dirsum lemma} for $p=\frac{d}{2}$ gives
\begin{equation}\label{equality1}
\begin{split}
	\lim_{t\to\infty}t^{\frac2d}\mu(t,\sum_{k=1}^K|A_kB_k|^2)
	&=\Big(\sum_{k=1}^K\Big(\kappa^2_d\Big\|\sum_{j\in C_k}\Gamma_j\otimes x_j\Big\|_{d}^2\|g_{m_k}\|_{d}^2\Big)^{\frac{d}{2}}\Big)^{\frac2d}\\
	&=\kappa^2_d\Big(\sum_{k=1}^K\Big\|\sum_{j\in C_k}\Gamma_j\otimes x_j\Big\|_{d}^d\|g_{m_k}\|_{d}^d\Big)^{\frac2d}\\
	&=\kappa^2_d\Big\|\sum_{k=1}^K\Big(\sum_{j\in C_k}\Gamma_j\otimes x_j\Big)\otimes g_{m_k}\Big\|_{d}^2=\kappa^2_d\Big\| \sum_{j=1}^J\Gamma_j\otimes x_j\otimes g_j \Big\|_{d}^2.
\end{split}
\end{equation}
In the third equality, we used the fact that  $\{g_{m_k}\}_{1\leq k\leq K}$ are pairwise disjointly supported.
Note that 
$$[A_k,B_k]=\sum_{j\in C_k}\Gamma_{j}\otimes [\pi_1(x_j),\pi_2(g_{m_k})(1-\Delta_{\theta})^{-\frac12}].$$
So by Corollary \ref{coro1} we have
\begin{equation}\label{equality2}
[A_k,B_k]\in(\mathcal{L}_{d,\infty})_0.
\end{equation}
By Lemma \ref{arbitrary index for Bessel} we have $A_kB_k\in\mathcal{L}_{d,\infty}.$ It then follows from the H\"{o}lder inequality that
$$|A_kB_k|^2-|B_kA_k|^2=(A_kB_k)^{\ast}[B_k,A_k]+[A_k,B_k]^{\ast}(A_kB_k)\in (\mathcal{L}_{\frac{d}{2},\infty})_0.$$
Summing over $1\leq k\leq K,$ we conclude that
\begin{equation}\label{estimate3}
\sum_{k=1}^K|A_kB_k|^2-\sum_{k=1}^K|B_kA_k|^2\in(\mathcal{L}_{\frac{d}{2},\infty})_0.
\end{equation}
Combining \eqref{equality1},  \eqref{estimate3} and Lemma \ref{bs sep lemma}, we arrive at $$\lim_{t\to\infty}t^{\frac2d}\mu(t,\sum_{k=1}^K|B_kA_k|^2)=\kappa^2_d\big\| \sum_{j=1}^J\Gamma_j\otimes x_j\otimes g_j \big\|_{d}^2.$$
By the pairwise orthogonality of $\{B_k\}_{k=1}^{K},$ we have
$|\sum_{k=1}^KB_kA_k|^2=\sum_{k=1}^K|B_kA_k|^2.$
In particular,
$$\mu^2(t,\sum_{k=1}^KB_kA_k)=\mu(t,\sum_{k=1}^K|B_kA_k|^2),\quad t>0.$$
Therefore,
$$\lim_{t\to\infty}t^{\frac2d}\mu^2(t,\sum_{k=1}^KB_kA_k)=\kappa^2_d\big\| \sum_{j=1}^J\Gamma_j\otimes x_j\otimes g_j \big\|_{d}^2.$$
That is,
$$
\lim_{t\to\infty}t^{\frac1d}\mu(t,\sum_{k=1}^KB_kA_k)=\kappa_d\big\| \sum_{j=1}^J\Gamma_j\otimes x_j\otimes g_j \big\|_{d}.
$$
Since by \eqref{equality2} we have
$$\sum_{k=1}^KA_kB_k=\sum_{k=1}^KB_kA_k+\sum_{k=1}^K[A_k,B_k]\in\sum_{k=1}^KB_kA_k+(\mathcal{L}_{d,\infty})_0,$$
the result now follows from Lemma \ref{bs sep lemma}.
\end{proof}

For general case, we shall borrow a splitting trick from \cite{SXZ2023}.
\begin{proof}[Proof of Theorem \ref{key thm plane}] For every $K\in\mathbb{N},$ let $(A_{K,k})_{k=1}^K\subset\mathbb{S}^{d-1}$ be pairwise disjoint open sets whose union is $\mathbb{S}^{d-1}$ (up to a set of measure $0$). We ask that
$$\max_{1\leq k\leq K}{\rm diam}(A_{K,k})\to0,\quad K\to\infty.$$	
Let $(\phi_{K,k})_{1\leq k\leq K}\subset C(\mathbb{S}^{d-1})$ be such that
\begin{enumerate}[(i)]
\item $0\leq \phi_{K,k}\leq 1$ for every $1\leq k\leq K;$
\item $\phi_{K,k}$ is supported at $A_{K,k}$ for $1\leq k\leq K;$
\item $ \displaystyle    \sum_{k=1}^Km(\{\phi_{K,k}\neq1\})\leq\frac1K\,$.
\end{enumerate}	
This can be done by taking $\phi_{K,k}\in C_c^{\infty}(A_{K,k})$ such that $\phi_{K,k}=1$ on an open subset $B_{K,k}\subseteq A_{K,k}$ with $m(B_{K,k}\backslash A_{K,k})\leq K^{-2}$ and then extending each $\phi_{K,k}$ from $A_{K,k}$ to $\mathbb{S}^{d-1}.$
Denote for brevity $T=\sum_{j=1}^J\Gamma_j\otimes\pi_1(x_j)\pi_2(g_j) (1-\Delta_{\theta})^{-\frac12},$ and set
$$g_{j,K,k}=\frac1{m(A_{K,k})}\int_{A_{K,k}}g_j,\quad g_{j,K}=\sum_{k=1}^Kg_{j,K,k}\phi_{K,k},$$
$$T_K=\sum_{j=1}^J\Gamma_j\otimes\pi_1(x_j)\pi_2(g_{j,K}) (1-\Delta_{\theta})^{-\frac12},\quad\phi_K=\sum_{k=1}^K\phi_{K,k}.$$
We have
$T(\mathrm{id}\otimes \pi_2(\phi_K))-T_K=\sum_{j=1}^J\Gamma_j\otimes\pi_1(x_j)(1-\Delta_{\theta})^{-\frac12}\pi_2(g_j\phi_K-g_{j,K}).$
By Lemma \ref{arbitrary index for Bessel}, $\pi_1(x_j)(1-\Delta_{\theta})^{-\frac12}\in\mathcal{L}_{d,\infty}$ for all $1\leq j\leq J.$ Hence, it follows from the quasi-triangle and H\"older inequalities, and Lemma \ref{matrix factor ignored} that
\begin{equation*}
\begin{split}
	&\|T(\mathrm{id}\otimes\pi_2(\phi_K))-T_K\|_{d,\infty}\\
	&\lesssim\sum_{j=1}^J\|\Gamma_j\|_d\|\pi_1(x_j)(1-\Delta_{\theta})^{-\frac12}\|_{d,\infty}\cdot\|\pi_2(g_j\phi_K-g_{j,K})\|_\infty\\
	&\lesssim \max_{1\leq j\leq J}\|g_j\phi_K-g_{j,K}\|_{\infty}.
\end{split}
\end{equation*} 
Denote for any $t>0$ the modulus of continuity
$$\omega(t,g)=\sup_{\substack{s_1,s_2\in\mathbb{S}^{d-1}\\ {\rm dist}(s_1,s_2)\leq t}}|g(s_1)-g(s_2)|.$$
This implies that for any open set $A\subseteq\mathbb{S}^{d-1},$ one has
$ |g(t)-g(s)|\leq \omega({\rm diam}(A),g),$ for all $s,t\in A.$
In particular, if $g\in C(\mathbb{S}^{d-1}),$ then $\lim_{t\to0}\omega(t,g)=0.$
A direct computation shows
$$\begin{aligned}
\|g_j\phi_K-g_{j,K}\|_{\infty} &=\|\sum_{k=1}^K(g_j-g_{j,K,k})\phi_{K,k}\|_{\infty}\\&=\|\sum_{k=1}^K\frac{1}{m(A_{K,k})}\int_{A_{K,k}}g_j-g_j(s)ds\cdot\phi_{K,k}\|_{\infty}\\
&\leq\|\sum_{k=1}^K\frac{1}{m(A_{K,k})}\int_{A_{K,k}}\omega\big(\max_{1\leq k'\leq K}{\rm diam}(A_{K,k'}),g_j\big)ds\cdot\phi_{K,k}\|_{\infty}
\\&\leq \omega\big(\max_{1\leq l\leq K}{\rm diam}(A_{K,k'}),g_j\big)\|\sum_{k=1}^K\phi_{K,k}\|_{\infty}\\
&\leq \omega\big(\max_{1\leq k'\leq K}{\rm diam}(A_{K,k'}),g_j\big)\to0,\quad K\to\infty.
\end{aligned}$$
Therefore,
\begin{equation}\label{ktt eq0}
\|T (\mathrm{id}\otimes\pi_2(\phi_K))-T_K\|_{d,\infty}\to0,\quad K\to\infty.
\end{equation}
Next, we write
$T(\mathrm{id}\otimes \pi_2(\phi_K))-T=\sum_{j=1}^J\Gamma_j\otimes\pi_1(x_j)\pi_2(\phi_K-1) (1-\Delta_{\theta})^{-\frac12}\pi_2(g_j).$
For any $(V_j)_{1\leq j\leq J}\subset(\mathcal{L}_{d,\infty})_0(L_2(\mathbb{R}^d)),$ it is clear that
$$ \sum_{j=1}^J\Gamma_j\otimes V_j\pi_2(g_j)\in(\mathcal{L}_{d,\infty})_0(\mathbb{C}^n\otimes L_2(\mathbb{R}^d)).$$ 
Thus,
\begin{equation*}
\begin{split}
	&{\rm dist}_{\mathcal{L}_{d,\infty}}(T (\mathrm{id}\otimes \pi_2(\phi_K))-T,(\mathcal{L}_{d,\infty})_0)\\
	&\leq \Big\|\sum_{j=1}^J\Gamma_j\otimes\pi_1(x_j)\pi_2(\phi_K-1) (1-\Delta_{\theta})^{-\frac12}\pi_2(g_j)-\sum_{j=1}^J\Gamma_j\otimes V_j\pi_2(g_j)\Big\|_{d,\infty}\\
	&\lesssim \sum_{j=1}^J\big\|\Gamma_j\big\|_d\cdot\big\|\pi_1(x_j)\pi_2(\phi_K-1) (1-\Delta_{\theta})^{-\frac12}-V_j\big\|_{d,\infty}\cdot\big\|\pi_2(g_j)\big\|_{\infty}\\
	&\lesssim\sum_{j=1}^J\|\pi_1(x_j)\pi_2(\phi_K-1) (1-\Delta_{\theta})^{-\frac12}-V_j\|_{d,\infty}.
\end{split}
\end{equation*}
Taking infimum over all $(V_j)_{1\leq j\leq J}\subset(\mathcal{L}_{d,\infty})_0,$ we arrive at
$$\begin{aligned}
&{\rm dist}_{\mathcal{L}_{d,\infty}}(T (\mathrm{id}\otimes \pi_2(\phi_K))-T,(\mathcal{L}_{d,\infty})_0)\\
&\lesssim\sum_{j=1}^{J}{\rm dist}_{\mathcal{L}_{d,\infty}}( \pi_1(x_j)\pi_2(\phi_K-1) (1-\Delta_{\theta})^{-\frac12},(\mathcal{L}_{d,\infty})_0).
\end{aligned}$$
By the distance formula \eqref{distance formula} and Theorem \ref{key thm}, we have
$${\rm dist}_{\mathcal{L}_{d,\infty}}( \pi_1(x_j)\pi_2(\phi_K-1) (1-\Delta_{\theta})^{-\frac12},(\mathcal{L}_{d,\infty})_0)=\kappa_d \|x_j\|_d\|\phi_K-1\|_d.$$
Consequently,
${\rm dist}_{\mathcal{L}_{d,\infty}}(T(\mathrm{id}\otimes  \pi_2(\phi_K))-T,(\mathcal{L}_{d,\infty})_0) \lesssim \|\phi_K-1\|_d.$
Hence,
\begin{equation}\label{ktt eq1}
{\rm dist}_{\mathcal{L}_{d,\infty}}(T (\mathrm{id}\otimes\pi_2(\phi_K))-T,(\mathcal{L}_{d,\infty})_0)\to0,\quad K\to\infty.
\end{equation}
Combining \eqref{ktt eq0} and \eqref{ktt eq1}, we obtain
\begin{equation}\label{ktt eq2}
{\rm dist}_{\mathcal{L}_{d,\infty}}(T-T_K,(\mathcal{L}_{d,\infty})_0)\to0,\quad K\to\infty.
\end{equation}
Now write
$$T_K=\sum_{j=1}^J\sum_{k=1}^Kg_{j,K,k}\Gamma_j\otimes \pi_1(x_j)\pi_2(\phi_{K,k}) (1-\Delta_{\theta})^{-\frac12}.$$
This goes back to the case of Lemma \ref{key thm disjoint plane}, and therefore
$$\lim_{t\to\infty}t^{\frac1d}\mu(t,T_K)=\kappa_d\|\sum_{j=1}^J\sum_{k=1}^Kg_{j,K,k}\Gamma_j\otimes x_j\otimes\phi_{K,k}\|_d,\quad K\in\mathbb{N}.$$
Since also \eqref{ktt eq2} holds, then applying Lemma \ref{another limit lemma} gives
$$\lim_{t\to\infty}t^{\frac1d}\mu(t,T)=\kappa_d\lim_{K\to\infty}\|\sum_{j=1}^J\Gamma_j\otimes x_j\otimes g_{j,K}\|_d.
$$
Noting that by the previous arguments,
$$\begin{aligned}
\|g_j-g_{j,K}\|_d&\leq m(\mathbb{S}^{d-1})^{\frac1d}\|g_j-g_{j,K}\|_\infty\\&\leq m(\mathbb{S}^{d-1})^{\frac1d}\Big(\|g_j-g_j\phi_K\|_{\infty}+\|g_j\phi_K-g_{j,K}\|_{\infty}\Big)\\
&\leq m(\mathbb{S}^{d-1})^{\frac1d}\Big(\|g_j\|_{\infty}\|1-\phi_K\|_{\infty}+\|g_j\phi_K-g_{j,K}\|_{\infty}\Big)\to0
\end{aligned}$$
as $K\to\infty,$
We conclude that
$$\sum_{j=1}^J\Gamma_j\otimes x_j\otimes g_{j,K}\to \sum_{j=1}^J\Gamma_j\otimes x_j\otimes g_j,\quad K\to\infty,$$
in $L_d(M_{n}(\mathbb{C})\bar{\otimes}L_{\infty}(\mathbb{R}_\theta^d)\bar{\otimes}L_{\infty}(\mathbb{S}^{d-1})).$ The result follows immediately.
\end{proof}	
\subsection{Proof of Theorem \ref{main1} }\label{sec-proof of main result1 and semi-norm rmk}
In this subsection, we will use Theorem \ref{key thm plane} to prove Theorem \ref{main1}.

Let $x\in\mathcal{S}(\qr),$ and set
$$\mathcal{A}=\sum_{j=1}^d\gamma_j\otimes A_j,\quad A_j=\pi_1(\partial_j x)-\sum_{k=1}^d\pi_2(\textbf{s}_j \textbf{s}_k)\pi_1(\partial_k x) ,\quad 1\leq j\leq d.$$
Here, each $\mathbf{s}_j\in C(\mathbb{S}^{d-1})$ is defined by setting $\mathbf{s}_j:(s_1,\cdots,s_d)\mapsto s_j.$

The next crucial fact is established in \cite[p.533]{MSX2020}.
\begin{lem}\label{msx 63 lemma}
If $x\in \mathcal{S}(\qr),$ then $\qd x-\mathcal{A}(1+\mathcal{D}^2)^{-\frac12}\in\mathcal{L}_{\frac{d}{2},\infty}.$
\end{lem}

\begin{proof}[Proof of Theorem \ref{main1}] Assume firstly that $x\in\mathcal{S}(\qr),$ and let $\mathcal{A}$ be defined as above. 
Note that Corollary \ref{coro1} we have
\small{$$
\begin{aligned}
(\gamma_j\otimes1)\mathcal{A}(1+\mathcal{D}^2)^{-\frac12}&=\gamma_j\otimes \pi_1(\partial_j x)(1-\Delta_{\theta})^{-\frac12} -\sum_{k=1}^d\gamma_j\otimes \pi_2(\textbf{s}_j \textbf{s}_k)\pi_1(\partial_k x)(1-\Delta_{\theta})^{-\frac12}\\
&=\gamma_j\otimes \pi_1(\partial_j x)(1-\Delta_{\theta})^{-\frac12} -\sum_{k=1}^d\gamma_j\otimes \pi_1(\partial_k x)\pi_2(\textbf{s}_j \textbf{s}_k)(1-\Delta_{\theta})^{-\frac12}
\end{aligned}
$$}
 modulo $(\mathcal{L}_{d,\infty})_0.$
Thus, by Theorem \ref{key thm plane} and equation \eqref{equivalet seminorm}, we obtain
$$\begin{aligned}
\lim_{t\to\infty}\mu(t,\mathcal{A}(1+\mathcal{D}^2)^{-\frac12})&=\kappa_d\Big\|\sum_{j=1}^d\gamma_j\otimes\partial_jx\otimes1-\sum_{j,k=1}^{d}\gamma_j\otimes\partial_kx\otimes\mathbf{s}_j\textbf{s}_k\Big\|_d\\
&=\kappa_d|||x|||_{\dot{W}^1_d(\qr)}.
\end{aligned}$$
Hence, by Lemmas \ref{bs sep lemma} and \ref{msx 63 lemma}, we deduce that
$
\lim_{t\to\infty}t^{\frac1d} \mu(t,\qd x)=\kappa_d|||x|||_{\dot{W}^1_d(\qr)},$ for $ x\in\mathcal{S}(\qr).
$	
Suppose now $x\in\dot{W}^{1}_d(\qr).$ We may find (see \cite[Theorem 4.8, Proposition 4.9]{MSX2023}) a sequence $\{x_n\}_{n=1}^\infty\subset\mathcal{S}(\qr)$ such that $\|x_n-x\|_{\dot{W}^1_d(\qr)}\to0$ and $
\|\qd x_n -\qd x\|_{d,\infty}\to0$, as$ n\to\infty.$ By the preceding paragraph, we have
$\lim_{t\to\infty}t^{\frac1d} \mu(t,\qd x_n)=\kappa_d|||x_n|||_{\dot{W}^1_d(\qr)},$ for all $n\geq1.$	
Appealing to Lemma \ref{another limit lemma}, we arrive at
$$\lim_{t\to\infty}t^{\frac1d} \mu(t,\qd x)=\kappa_d\lim_{n\to\infty}|||x_n|||_{\dot{W}^1_d(\qr)}.$$ Since semi-norms $|||\cdot|||_{\dot{W}^1_d(\qr)}$ and $\|\cdot\|_{\dot{W}^1_d(\qr)}$ are equivalent (see Section \ref{equivalence}), the right hand side is exactly $\kappa_d|||x|||_{\dot{W}^1_d(\qr)}.$ This completes the proof.
\end{proof}	
\subsection{Proof of Theorem \ref{main2}}\label{sec-proof of main2}In this subsection, we will deduce Theorem \ref{main2} from Theorem \ref{key thm plane}. In order to apply Theorem \ref{key thm plane}, we need the following equivalences of rearrangement of products.  
\begin{lem}\label{rearrange}
Let $\{g_j\}_{j=1}^N\subset C(\mathbb{S}^{d-1}).$ 
\begin{enumerate}[(i)]
\item \label{cpt} If $\{x_j\}_{j=1}^N\subset C_0(\qr)+\mathbb{C},$  then
\small{$$
	\prod_{j=1}^N\pi_1(x_j)\pi_2(g_j)  -\pi_1\big(\prod_{j=1}^N x_j\big)\pi_2\big(\prod_{j=1}^Ng_j\big)\in \mathcal{K}.
	$$}
\item\label{ld0} If $\{x_j\}_{j=1}^N\subset \mathcal{S}(\qr)+\mathbb{C},$ then
\small{$$
	\big(\prod_{j=1}^N\pi_1(x_j)\pi_2(g_j)\big)(1-\Delta_{\theta})^{-\frac12}-\pi_1\big(\prod_{j=1}^N x_j\big)\pi_2\big(\prod_{j=1}^Ng_j\big)(1-\Delta_{\theta})^{-\frac12}\in (\mathcal{L}_{d,\infty})_0$$.}
\end{enumerate}
\end{lem}
\begin{proof}
The proof of the first statement is extremely identical to that of \cite[Lemma 5.1]{SXZ2023} and \cite[Lemma 4.18]{FSYZ2024}, and hence it is omitted. The second statement can be handled similarly, the key point is Theorem \ref{commutator-PSD} and the following identity
\small{\begin{equation*}
		\begin{split}
			&\prod_{j=1}^N\pi_1(x_j)\pi_2(g_j)  -\pi_1\big(\prod_{j=1}^N x_j\big)\pi_2\big(\prod_{j=1}^Ng_j\big)\\
			&=\sum_{j=1}^N\pi_1(x_1)\pi_2(g_1)\cdots\pi_1(x_{j-1})\pi_2(g_{j-1})\pi_1(x_j)[\pi_2(g_j),\pi_1(x_{j+1})]\pi_2(g_{j+1})\cdots\pi_2(g_N).
		\end{split}
\end{equation*}}
\end{proof}

We are finally in position to prove our second main result.
\begin{proof}[Proof of Theorem \ref{main2}]
Recall that by assumption, $T$ is in the $C^{\ast}$-algebra generated by $\pi_{1}(C_0(\qr)+\mathbb{C})$ and $\pi_2(C(\mathbb{R}^d)),$ satisfying $T=T\pi_1(Z)$ for some $Z\in\mathcal{S}(\qr).$
Since by Proposition \ref{density} the algebra $\mathcal{S}(\qr)$ is norm-dense in $C_0(\qr),$ we may find a sequence $\{T_n\}_{n=1}^\infty$ in the algebra generated by $\pi_1(\mathcal{S}(\qr)+\mathbb{C})$ and $\pi_2(C(\mathbb{S}^{d-1}))$ such that $T_n\to T$ in the operator norm. Each $T_n$ can be expressed as
$$T_n=\sum_{k=1}^{k_n}   \prod_{j=1}^{j_n}\pi_1(x_{j,k,n})\pi_2(g_{j,k,n}),$$
with $x_{j,k,n}\in\mathcal{S}(\qr)+\mathbb{C}$ and $g_{j,k,n}\in C(\mathbb{S}^{d-1}).$
Set
$$y_{k,n}=\big(\prod_{j=1}^{j_n} x_{j,k,n}\big)\cdot Z,\quad h_{k,n}=\prod_{j=1}^{j_n} g_{j,k,n},\quad S_n=\sum_{k=1}^{k_n}\pi_1(y_{k,n})\pi_2(h_{k,n}).$$
It follows from Lemma \ref{rearrange} that
\begin{equation}\label{equivalence modulo compact operators}
T_n\pi_1(Z)-S_n\in\mathcal{K}
\end{equation}
and that
$
T_n\pi_1(Z)(1-\Delta_{\theta})^{-\frac12}-S_n(1-\Delta_{\theta})^{-\frac12}\in (\mathcal{L}_{d,\infty})_0.
$
Since ${\rm sym}(S_n)=\sum_{k=1}^{k_n}y_{k,n}\otimes h_{k,n},$ it follows from Lemma \ref{bs sep lemma} and Theorem \ref{key thm plane} that
$$
\begin{aligned}
\lim_{t\to\infty}t^{\frac1d}\mu(t,T_n\pi_1(Z)(1-\Delta_{\theta})^{-\frac12})&=\lim_{t\to\infty}t^{\frac1d} \mu(t,S_n(1-\Delta_{\theta})^{-\frac12})\\
&=\kappa_d\big\|\sum_{k=1}^{k_n}y_{k,n}\otimes h_{k,n}\big\|_d=\kappa_d\|{\rm sym}(S_n)\|_d.
\end{aligned}
$$
Note that by \eqref{equivalence modulo compact operators} and property \eqref{compact vanishing}, we have
${\rm sym}(S_n)={\rm sym}(T_n\pi_1(Z)).$
Therefore,
\begin{equation*}
\lim_{t\to\infty}t^{\frac1d}\mu(t,T_n\pi_1(Z)(1-\Delta_{\theta})^{-\frac12})=\kappa_d\|{\rm sym}(T_n\pi_1(Z))\|_d,\quad n\geq1. 
\end{equation*}
By Lemma \ref{arbitrary index for Bessel}, $\pi_1(Z)(1-\Delta_{\theta})^{-\frac12}\in\mathcal{L}_{d,\infty}.$ Since $T_n\to T$ in the operator norm, it is clear that
$T_n\pi_1(Z)(1-\Delta_{\theta})^{-\frac12}\to T\pi_1(Z)(1-\Delta_{\theta})^{-\frac12}=T(1-\Delta_{\theta})^{-\frac12}$ in $\mathcal{L}_{d,\infty}.$ Hence, the conditions of Lemma \ref{another limit lemma} are satisfied. Applying Lemma \ref{another limit lemma}, the following limits exist and are equal:
$$\lim_{t\to\infty}t^{\frac1d}\mu(t,T(1-\Delta_{\theta})^{-\frac12})=\kappa_d\cdot\lim_{n\to\infty}\|{\rm sym}(T_n\pi_1(Z))\|_d.$$
Noting that, $\mathrm{sym}$ is a (norm) continuous $\ast$-homomorphism, and $\mathrm{sym}(\pi_1(Z))=Z\otimes 1\in L_d(L_\infty(\qr)\bar{\otimes}L_\infty(\mathbb{S}^{d-1})),$ we obtain
$${\rm sym}(T_n\pi_1(Z))\to{\rm sym}(T\pi_1(Z))={\rm sym}(T)\quad \mbox{in} \quad L_d(L_\infty(\qr)\bar{\otimes}L_\infty(\mathbb{S}^{d-1})).$$
This completes the proof of Theorem \ref{main2}.
\end{proof}

\subsection{Equivalence of semi-norms on $\dot{W}_d^1(\qr)$}\label{equivalence}
\begin{lem}\label{easy lemma} Let $\{T_j\}_{j=1}^d\subset\mathcal{L}_d.$ We have
$\|\sum_{j=1}^d\gamma_j\otimes T_j\|_d\geq\sqrt{2}\max_{1\leq j\leq d}\|T_j\|_d.$
\end{lem}
\begin{proof}
Indeed, by the anti-commutation relations of the $\gamma_j^{,}$s, we have
$$
2(\mathrm{id}\otimes T_k)=(\gamma_k\otimes1)\cdot\Big(\sum_{j=1}^{d}\gamma_j\otimes T_j\Big)+\Big(\sum_{j=1}^{d}\gamma_j\otimes T_j\Big)\cdot(\gamma_k\otimes1),\quad 1\leq k\leq d,
$$
it follows that
$ \sqrt{2}\|T_k\|_d=\|\mathrm{id}\otimes T_k\|_d\leq\|\sum_{j=1}^{d}\gamma_j\otimes T_j\|_d,$ for all 
$1\leq k\leq d.$ 
\end{proof}
In what follows, set $\omega_d=\frac{2\pi^{\frac{d}{2}}}{\Gamma(\frac{d}{2})},$ i.e., the area of unit sphere $\mathbb{S}^{d-1}.$
\begin{prop}\label{equivalence of seminorms}
There exist positive constants $c_d,C_d$ such that
$$c_d\|x\|_{\dot{W}_d^1(\qr)}\leq |||x|||_{\dot{W}_d^1(\qr)}\leq C_d\|x\|_{\dot{W}_d^1(\qr)},\quad \forall x\in \dot{W}_d^1(\qr).$$ 
\end{prop}
\begin{proof}
By equation \eqref{equivalet seminorm} and the triangle inequality, we have
\begin{equation*}
\begin{split}
	|||x|||_{\dot{W}_d^1(\qr)}&\leq\sqrt{2}\sum_{j=1}^d\big\|\partial_jx\otimes 1-\sum_{k=1}^d\partial_kx\otimes \textbf{s}_k\textbf{s}_j\big\|_{L_d(L_{\infty}(\mathbb{R}_\theta^d)\bar{\otimes} L_\infty(\mathbb{S}^{d-1}))}\\
	&\leq\sqrt{2}\Big(\sum_{j=1}^d\big\|\partial_j x\|_{L_d(\mathbb{R}_\theta^d)}\cdot\omega^{\frac1d}_d+\sum_{k,j=1}^d\big\|\partial_k x\|_{L_d(\mathbb{R}_\theta^d)}\|\textbf{s}_k\textbf{s}_j\|_{L_d(\mathbb{S}^{d-1})}\Big)\\
	&\leq\sqrt{2}\sum_{j=1}^d\big\|\partial_j x\|_{L_d(\mathbb{R}_\theta^d)}\omega^{\frac1d}_d(1+d)=\sqrt{2}\omega^{\frac1d}_d(1+d)\|x\|_{\dot{W}^{1}_d(\qr)}.
\end{split}
\end{equation*}
This gives the upper estimate. Next, by Lemma \ref{easy lemma} we have
$$
\begin{aligned}
|||x|||_{\dot{W}_d^1(\qr)}
&\geq\sqrt{2}\max_{1\leq j\leq d}\big\|\partial_jx\otimes 1-\sum_{k=1}^d\partial_kx\otimes \textbf{s}_k\textbf{s}_j\big\|_{L_d(L_{\infty}(\mathbb{R}_\theta^d)\bar{\otimes} L_\infty(\mathbb{S}^{d-1}))}
\\&=\sqrt{2}\max_{1\leq j\leq d}\big\|\partial_j x-\textbf{s}_j\sum_{k=1}^d \textbf{s}_k  \partial_k x\big\|_{L_d(\mathbb{S}^{d-1},L_d(\qr))}.
\end{aligned}
$$
Observe that for all $f\in L_d(\mathbb{S}^{d-1},L_d(\qr)),$ one has
$$ \|f\|_{L_d(\mathbb{S}^{d-1},L_d(\qr))}:=\Big(\int_{\mathbb{S}^{d-1}}\big\|f(s)\big\|_{L_d(\qr)}^d ds\Big)^{\frac1d}\geq\omega_d^{\frac1d-1}\Big\|\int_{\mathbb{S}^{d-1}}f(s) ds\Big\|_{L_d(\qr)}.$$
Based on this observation, we proceed as
\begin{equation*}
\begin{split}
	|||x|||_{\dot{W}_d^1(\qr)}
	&\geq \sqrt{2}\omega_d^{\frac1d-1}\max_{1\leq j\leq d}\big\|\int_{\mathbb{S}^{d-1}}\Big(\partial_j x-\textbf{s}_j\sum_{k=1}^d \textbf{s}_k  \partial_k x\Big)ds\big\|_{L_d(\qr)}\\
	&=\sqrt{2}\omega_d^{\frac1d-1}\max_{1\leq j\leq d}\big\|\partial_j x\big\|_{L_d(\qr)}\cdot\int_{\mathbb{S}^{d-1}}(1-\textbf{s}_j^2)ds\\
	&=\sqrt{2}\omega_d^{\frac1d}(1-\frac1d)\max_{1\leq j\leq d}\big\|\partial_j x\big\|_{L_d(\qr)}\geq \sqrt{2}\frac{d-1}{d^2}\omega_d^{\frac1d}\|x\|_{\dot{W}^{1}_d(\qr)}.
\end{split}
\end{equation*}
Finally, setting $c_d=\sqrt{2}\frac{d-1}{d^2}\omega_d^{\frac1d}$ and $C_d=\sqrt{2}\omega^{\frac1d}_d(1+d)$ completes the proof.
\end{proof}

\section*{Declarations}

\textbf{Conflict of interest} The author declares that there is no conflict of interest.

\noindent\textbf{Data availability} No data was used for the research described in this article.

\bibliography{sn-bibliography}

\end{document}